\documentclass[12pt,reqno]{amsart}
\usepackage{comment}
\usepackage[utf8]{inputenc} 
\usepackage[T1]{fontenc} 
\usepackage{geometry}
\usepackage[dvipsnames]{xcolor}
\usepackage[normalem]{ulem}

\usepackage{amsfonts}
\usepackage{amssymb}
\usepackage{amsthm}
\usepackage{amsmath}
\usepackage{amscd}
\usepackage[shortlabels]{enumitem}
\usepackage{mathrsfs}
\usepackage{tikz}
\usetikzlibrary{calc,arrows,decorations.pathreplacing, automata, positioning}
\usepackage{nicefrac, xfrac}
\usepackage{mathtools}
\usepackage[bottom]{footmisc}
\usepackage[mathcal]{euscript}
\usepackage{float}
\usepackage{ stmaryrd }
\usepackage[sort]{cite}
\usepackage[pagebackref, colorlinks = true,
linkcolor = red,
urlcolor  = blue,
citecolor = blue,
anchorcolor = blue]{hyperref}
\hypersetup{pdftitle={\@title},pdfauthor={\@author}}

\theoremstyle{plain}

\newtheorem{theorem}{Theorem}[section]
\newtheorem{corollary}[theorem]{Corollary}

\newtheorem{lemma}[theorem]{Lemma}
\newtheorem{example}[theorem]{Example}

\newtheorem{proposition}[theorem]{Proposition}

\theoremstyle{definition}
\newtheorem{definition}[theorem]{Definition}

\newtheorem{remark}[theorem]{Remark}

\newtheorem{thmx}{Theorem}

\renewcommand{\phi}{\varphi}
\let\oldepsilon\epsilon 
\renewcommand{\epsilon}{\varepsilon}

\newcommand{\Om}{\ensuremath{\Omega}}

\newcommand{\bt}{\ensuremath{\beta}}  
\newcommand{\sg}{\ensuremath{\sigma}}

\newcommand{\kp}{\ensuremath{\kappa}}

\DeclareSymbolFont{bbold}{U}{bbold}{m}{n}
\DeclareSymbolFontAlphabet{\mathbbold}{bbold}

\newcommand{\cA}{\ensuremath{\mathcal{A}}}

\newcommand{\cG}{\ensuremath{\mathcal{G}}}

\newcommand{\cL}{\ensuremath{\mathcal{L}}}
\newcommand{\cM}{\ensuremath{\mathcal{M}}}

\newcommand{\cS}{\ensuremath{\mathcal{S}}}

\newcommand{\bN}{\mathbb{bN}}
\newcommand{\bR}{\mathbb{R}}
\newcommand{\bQ}{\mathbb{Q}}

\newcommand{\bZ}{\mathbb{Z}}

\newcommand{\NN}{\ensuremath{\mathbb N}}

\newcommand{\QQ}{\ensuremath{\mathbb Q}}
\newcommand{\RR}{\ensuremath{\mathbb R}}

\newcommand{\ZZ}{\ensuremath{\mathbb Z}} 

\newcommand{\mumax}{\ensuremath{\mu_{\rm max}}} 
\newcommand{\htop}{\ensuremath{h_{\rm top}}}
\newcommand{\hcon}{\ensuremath{h_{\rm con}}}
\newcommand{\hres}{\ensuremath{h_{\rm res}}}
\newcommand{\Xcon}{\ensuremath{X_{\rm con}}}
\newcommand{\Xres}{\ensuremath{X_{\rm res}}}
\newcommand{\Xlim}{\ensuremath{X_{\rm lim}}}

\usepackage{mathtools}

\newcommand{\ceil}[1]{\ensuremath{\left\lceil#1 \right\rceil}}
\newcommand{\floor}[1]{\ensuremath{\left\lfloor#1 \right\rfloor}}

\providecommand{\phantomsection}{}
\AtBeginDocument{\let\textlabel\label}
\makeatletter
\newcommand{\mylabel}[2]{\raisebox{.7\normalbaselineskip}{\phantomsection}(#1)%
	\def\@currentlabel{#1}\textlabel{#2}}
\makeatother

\makeatletter
\newcommand\xlabel[2][]{\phantomsection\def\@currentlabelname{#1}\label{#2}}
\makeatother

\allowdisplaybreaks

\numberwithin{equation}{section}
\title[]{Computability of $\cG$-Beroulli Measures and Measures of Maximal Entropy on Coded Shift Spaces}
\date{\today}

\author[T. Kucherenko]{Tamara Kucherenko}
\address{City University of New York, New York, NY, USA}
\email{\href{tkucherenko@ccny.cuny.edu} {tkucherenko@ccny.cuny.edu} }
\thanks{T.K. was partially supported by the grant from the Simons Foundation (SF-MPS-CGM-855117).}

\author[M. L\'opez]{Marco L\'opez}
\address{City University of New York, New York, NY, USA}
\email{\href{mlopez1@gc.cuny.edu}{mlopez1@gc.cuny.edu} }
\thanks{M.L. was partially supported by the National Science Foundation through Grant DMS-2402751.}
\author[C. Wolf]{Christian Wolf}
\address{Department of Mathematics and Statistics\\
Mississippi State University\\
Starkville, MISSISSIPPI, 39759, USA}
\email{\href{cwolf@math.msstate.edu} {cwolf@math.msstate.edu}}
\thanks{C.W. was partially supported by grants from the Simons Foundation (SF-MPS-CGM-637594 and SFI-MPS-TSM-00013897).}

\begin{document}
\maketitle
	
\begin{abstract}

In this paper, we investigate the computability of $\cG$-Bernoulli measures, with a particular focus on measures of maximal entropy (MMEs) on coded shift spaces. Coded shifts are natural generalizations of sofic shifts and are defined as the closure of all bi-infinite concatenations of words (generators) drawn from a countable generating set $\cG$.
We begin by establishing a computability criterion for $\cG$-Bernoulli measures which are invariant measures given by assigning probability weights to the generators. We then apply this criterion to the setting in which the concatenation entropy exceeds the residual entropy, showing that in this case the unique measure of maximal entropy $\mu_{\rm max}$ on $X$ is computable, provided the Vere–Jones parameter $\kappa$ of $\cG$ is computable, based on having oracle access to the generators and the language of $X$.
As a consequence, the unique MME is computable for several well-known classes of shift spaces, including $S$-gap shifts, multiple-gap shifts, and $\beta$-shifts. Moreover, the two ergodic MMEs of the Dyck shift are also computable.
Finally, we examine the opposite situation, where the residual entropy exceeds the concatenation entropy and the MME is known to be non-unique in general. We show that even when $\mu_{\rm max}$ is unique and the parameter $\kappa$ is computable, the measure $\mu_{\rm max}$ may still fail to be computable.

\end{abstract}


\section{Introduction}
Computers are a powerful tool to analyze models of natural phenomena. Heuristically, it is widely believed that computer simulations can predict the correct qualitative behavior in many applications. For dynamical systems this belief is often based on the shadowing property of pseudo-orbits \cite{kruger_complexity_1994, boffetta_predictability_2002, Skeel, gora_why_1988}. It is of practical interest to establish guarantees that numerical approximations and simulations adequately represent the mathematical object being approximated or modeled. In essence, a mathematical object (e.g. a number, a function, a set, a measure, etc.) is computable if there is a computer algorithm producing a sequence of outputs converging to the said object \emph{and} the error bounds decrease at a prescribed rate. The latter is required so that the algorithm also provides sufficient criteria to determine how far along it needs to compute the sequence to obtain an approximation within a specified error tolerance. For the last two decades, much effort has been directed towards computability results for dynamical quantities that would guarantee that simulations produce typical dynamical behavior \cite{burr_computability_2022, burr_computability_2024}, although this line of work extends further back \cite{benettin_reliability_1979, kruger_complexity_1994}. Two such dynamical quantities are entropy and measures of maximal entropy, the latter being the focus of this paper.

\subsection{Entropy and Measures of Maximal Entropy}
 Conceptually, the entropy of a dynamical system provides a way to quantify the exponential rate at which measurement errors will propagate over time. Such a concept can be precisely defined in either a topological or measure-theoretic setting, the two being related by a variational principle. The measure-theoretic (i.e., probabilistic) side was initially developed by Shannon \cite{shannon_mathematical_1948} in the context of sequences of symbols in information theory and was later extended to symbolic dynamics by Kolmogorov \cite{kolmogorov_1958, kolgomorov_entropy_1959}. The monograph by Dajani and Kraaikamp \cite{dajani_ergodic_2002} provides a good introduction to these concepts.

The Shannon and Kolmogorov entropies are the properties of a measure, and measures with maximal entropy are of particular interest. Such measures represent an equilibrium probability distribution and they are the result of a widely used principle to approximate unknown distributions. This maximum entropy principle, due to Jaynes, states that the best approximation should satisfy any known constraints while maximizing entropy, which can also be interpreted as being as close to a uniform distribution as possible while respecting known constraints \cite{jaynes_information_1957}. Thus, finding an MME solves a relevant optimization problem, and this turns out to be useful in a number of applied settings, from modeling species distributions in ecology to neuroscience (see \cite{Phillips, Tang} and the references therein for examples of such applications). 

\subsection{Coded shifts}\label{sec:coded_shifts}
Coded shifts provide a fairly general setting in symbolic dynamics. These shift spaces were introduced by Blanchard and Hansel \cite{blanchard_systemes_1986}, and are a natural generalization to irreducible  sofic and synchronized shifts. They include many classes of examples such as $S$-gap shifts \cite{burr_computability_2022}, $\beta$-shifts \cite{dajani_ergodic_2002}, and its generalizations, see, e.g., \cite{thompson_generalized_2017, burr_computability_2022}. Recall that all finite-alphabet shift spaces are associated with a countable (possibly infinite) directed labeled graph, which is called a representation of the shift space. In this construction, the shift space is given as the closure of the set of bi-infinite sequences of symbols associated with all bi-infinite paths on the graph \cite{lind_introduction_2021}. Coded shifts are precisely those shift spaces that have a strongly connected representation. An alternative way to define coded shifts  is by taking a countable set of words  $\cG=\{g_i: i\in\NN\},$ called a generating set, and considering the {\bf concatenation set} $X_{\rm con}$ of bi-infinite concatenations of generators, that is: 
\[
\Xcon=\{\cdots g_{i_{-1}}g_{i_0}g_{i_1}\cdots\colon i_k\in\NN \}. 
\]
The coded shift $X=X(\cG)$ is then the closure of the concatenation set, whereas the {\bf residual set} of $X$ is defined by $\Xres=X\backslash\Xcon$.  We refer to \cite{kucherenko_ergodic_2024} for more details about coded shifts. If the generating set is finite we can make the same construction with a finite index set. We note that in the finite generating set case the coded shift $X$ is sofic and $X_{\rm res}=\emptyset$ in which case the computability of the unique MME is straightforward. Therefore, we focus in this paper on the infinite generating set case.

Two distinct generating sets $\cG$ and $\cG'$ may generate the same coded shift. For example, if we let $g_1=0,$ $g_2=01,$ and $g_3=00$ then $\cG=\{g_1,g_2\}$ and $\cG'=\{g_1,g_2,g_3\}$ generate the same coded shift. However, notice that the sequence $0^\infty,$ which is in $\Xcon$ has many distinct representations under $\cG',$ (e.g., $\cdots g_1g_3g_1g_3\cdots$ and $\cdots g_1g_1g_1\cdots$) but only one representation under $\cG.$ If every sequence in $\Xcon$ has a unique representation under $\cG$ we say that $\cG$ {\bf uniquely represents} $X_{\rm con}$ (also known as \emph{unambiguous} in \cite{beal_unambiguously_2024}, or \emph{uniquely decomposable} in \cite{pavlov_entropy_2020}). Unique representability of $X_{\rm con}$ is an essential condition for many results about coded shifts. It has been recently established in \cite{beal_unambiguously_2024} that every coded shift admits a generating set that uniquely represents the corresponding concatenation set, and the proof there provides a theoretical framework for constructing such a generating set. In fact, this construction has now been fully implemented algorithmically in \cite{kucherenko_Tupper}, which gives an explicit procedure that takes any generating set $\mathcal{G}'$ as input and outputs a generating set $\mathcal{G}$ with unique representability property and such that $X(\mathcal{G}) = X(\mathcal{G}')$.
 In this paper, we use as a standing assumption that $X=X(\cG)$ is a coded shift such that the generating set $\cG$ uniquely represents $X_{\rm con}$.
We refer to Section \ref{definitions} for relevant definitions and properties of coded shifts.

Our work focuses on classifying the computability of MMEs for coded shifts. Despite being irreducible, coded shifts may not have a unique MME, as shown in \cite{krieger_uniqueness_1974} for the Dyck shift and in \cite{kucherenko_ergodic_2024} for several other examples. Even when the MME is unique, it may not be computable from the language and the generators as inputs (see Theorem \ref{thmnoncomput}.) A key sufficient condition that allows us to derive computability results is a relationship between the entropies of the sets $\Xcon$ and $\Xres$, defined as follows:
\begin{equation}\label{eq1}\hcon(X) = \sup\{h_\sg(\mu)\colon \mu(\Xcon)=1, \mu\in\cM_\sg(X)\}, \end{equation}
and
\begin{equation}\label{eq2}\hres(X) = \sup\{h_\sg(\mu)\colon \mu(\Xres)=1, \mu\in\cM_\sg(X)\},\end{equation}
where $\cM_\sigma(X)$ is the set of shift-invariant Borel probability measures on $X$ and $h_\sg(\mu)$ denotes the metric entropy of $\mu.$ We note that these quantities are well-defined since $X_{\rm con}$ and $X_{\rm res}$ are Borel sets on $X$, see \cite{kucherenko_ergodic_2024}.
In \cite{kucherenko_ergodic_2024} the authors establish that $\hcon(X)>\hres(X)$ is a sufficient condition for the uniqueness of the MME. Moreover, they show that such measures satisfy a property called $\cG$-Bernoulli, in reference to the similarities with Bernoulli measures. More precisely, an invariant probability measure $\mu$ is $\cG$-Bernoulli if there exist positive numbers $p_g$ with $\sum_{g\in\cG}p_g=1$ and $c=\sum_{g\in\cG}|g|p_g<\infty$ such that on sets (called $\cG$-cylinders) of the form
\[
\llbracket g_{i_0}\cdots g_{i_k} \rrbracket =\{\cdots g'_{i_{-1}}.g'_{i_{0}}g'_{i_{1}}\cdots \in X_{\rm con}\colon g'_{i_j}=g_{i_j}\,\, {\rm for}\,\,\, j=0,\dots,k \}
\]
the measure $\mu$ is given by 

\begin{align}\label{measure}
\mu(\llbracket g_{i_1}\cdots g_{i_k}\rrbracket)=\frac{1}{c}p_{g_{i_1}}\cdots p_{g_{i_k}}.    
\end{align}
Here $c$ is the normalizing constant assuring that $\mu$ is a probability measure.
 There is a strongly connected directed labeled graph $\Gamma$ associated with a coded shift generated by $\cG.$ The edge shift of $\Gamma$ embeds naturally as a subsystem of $(X,\sigma)$. Vere-Jones \cite{Vere-Jones} introduced the following classification of $\Gamma:$
    \begin{itemize}
        \item $\Gamma$ is transient if $\sum_{g\in\cG}\exp(-|g|h(\Gamma))<1,$
        \item $\Gamma$ is null-recurrent if $\sum_{g\in\cG}\exp(-|g|h(\Gamma))=1$ and $\sum_{g\in\cG}|g|\exp(-|g|h(\Gamma))=\infty,$ and 
        \item $\Gamma$ is positively-recurrent if $\sum_{g\in\cG}\exp(-|g|h(\Gamma))=1$ and $\sum_{g\in\cG}|g|\exp(-|g|h(\Gamma))<\infty,$
    \end{itemize}
where $h(\Gamma)$ is the Gurevich entropy associated to $\Gamma$. In \cite{gurevich_shift_1970}, Gurevich showed that  positive recurrence is equivalent to the existence and uniqueness of an invariant Borel probability measure with metric entropy $h(\Gamma)$ (that is, an MME on the graph $\Gamma.$) When a coded shift $X$ corresponds to a positively-recurrent graph and the additional condition $\hcon(X)>\hres(X)$ is satisfied then $X$ admits a unique MME of the form \eqref{measure}, where $c=\kp=\kp(\cG)=\sum_{g\in\cG}|g|\exp(-|g|h_{\rm con}(X))$ is the parameter appearing in Vere-Jones' classification \cite{kucherenko_ergodic_2024}. We refer to the constant $\kp$ as the {\it Vere-Jones} parameter of the generating set $\cG$.

\subsection{Computability}

Various computability results have been established for shifts of finite type (including higher dimensional ones) \cite{burr_computability_2018, hedges_computability_2025} and sofic shifts \cite{spandl}. Recently, the computability of the topological pressure was characterized for coded shifts and continuous potentials in \cite{burr_computability_2022}.
Informally, a mathematical object --- for example a real number, a set, or a measure --- is said to be computable if one can design a Turing machine (a computer program for our purposes) capable of producing approximations of the object with arbitrarily high accuracy. Here, the required accuracy is part of the input to the computer program and the program must have the ability to halt once the computation has reached the required level of  approximation. We refer to Section \ref{definitions} for precise definitions. The computability framework for measures has been laid out in previous work, such as in \cite{burr_computability_2018, binder_computability_2011, hoyrup_computability_2009}.
In this article, we focus our attention on the computability of MMEs for coded shifts. 

\subsection{Statement of the results}
To formulate our main results we first introduce some additional notation. We refer to Section \ref{definitions} for details.
    We say a real number $c$ is computable if there exists a Turing machine (a computer program for our purposes) $T=T(n)$ which on input $n\in \bN$, outputs a rational number $T(n)$ such that $|T(n)-c|<2^{-n}$. Moreover, we say a sequence $(p_i)_{i\in \bN}$ of real numbers is uniformly computable if there exists a Turing machine $T=T(n,i)$ which on input $n,i\in \bN$, outputs a rational number $T(n,i)$ such that $|T(n,i)-p_i|<2^{-n}$. We stress that for uniform computability there needs to exist one Turing machine which approximates
all $p_i$ at any described precision. In particular, it is not sufficient that each $p_i$ is computable. 

Let $\cA=\{0,\dots,d-1\}$ be a finite alphabet, and let $\cA^*=\bigcup_{n=1}^{\infty}\cA^n$ denote the set of finite words over $\cA$. Let $X$ be a shift space over $\cA$, i.e., a closed shift-invariant subset of $\cA^\bZ$ (see Section 2.2 for details). The language $\cL(X)$ of $X$ is the set of all words  $w\in \cA^*$ such that $w$ occurs in some $x\in X$. By an oracle for the language $\mathcal{L}(X)$ we mean an enumeration of all words in the language, listed in nondecreasing order of length. Similarly, an oracle for a generating set $\mathcal{G}$ of a coded shift $X = X(\mathcal{G})$ is an enumeration of the generators, again ordered by nondecreasing word length, see Definition \ref{def:oracles_of_G_and_L(X)}. Assume that $\cG$ is given by an oracle and let $\mu$ be a $\cG$-Bernoulli measure given by positive real numbers $p_g$, see equation \eqref{measure}. We recall from \eqref{measure} that the normalization constant is given by
$$c=c(\cG,(p_i)_i)=\sum_{i=1}^\infty |g_i|p_i<\infty.$$ 
For $\ell\in \bN$ set $c_\ell=\sum_{i=1}^\ell |g_i|p_i$.
We say $c$ is $\cG$-computable if there exists a Turing machine $T=T(n)$, which based on  input $n\in \bN$, outputs $\ell\in \bN$ such that $| c-c_{\ell}|< 2^{-n}$. We note that the uniform computability of $(p_i)_i$ implies that $(c_\ell)_\ell$ is  uniformly computable.
With these definitions in mind, we are now able to formulate our first main result which concerns the computability of $\cG$-Bernoulli measures:

\begin{thmx}\label{thmmain} Let $X = X(\cG)$ be a coded shift with generating set $\cG$ such that $\cG$ uniquely represents $X_{\rm con}(\cG)$. Let $\mu$ be a $\cG$-Bernoulli measure given by positive real numbers $p_g$ with normalizing constant $c<\infty$, see equation \eqref{measure}. Suppose $c$ is $\cG$-computable and $(p_g)_{g\in \cG}$ is uniformly computable. Then $\mu$ is computable based on having oracle access to $\cG$. 
\end{thmx}

Next we apply Theorem \ref{thmmain} to coded shifts $X=X(\cG)$ satisfying $h_{\rm con}(X)>h_{\rm res}(X)$. For this case it is shown in \cite{kucherenko_ergodic_2024} that the (unique) measure of maximal entropy is $\cG$-Bernoulli with $p_g=\exp(-|g|\htop(X))$ and $c=\kp=\sum_{g\in \cG} |g| \exp(-|g|\htop(X))$. 

\begin{corollary}\label{cormain11} Let $X = X(\cG)$ be a coded shift with generating set $\cG$ such that $\cG$ uniquely represents $X_{\rm con}(\cG)$ and $\hcon(X)>\hres(X)$. If the Vere-Jones parameter $\kp(\cG)$ is $\cG$-computable, then the (unique) MME $\mu_{\rm max}$ of $X$ is computable given oracle access to $\mathcal{G}$ and  $\cL(X)$.
\end{corollary}

The corollary shows that the computability of the unique MME depends upon the $\cG$-computability of the Vere-Jones parameter. We show in Example \ref{ex:Non-computability_VJP} that the parameter $\kp$ is, in general, not $\cG$-computable without additional information about the coded shift and the associated generating set. Nevertheless, the next result provides a criterion that is easy to verify for many common classes of coded shifts, guaranteeing the computability of the Vere-Jones parameter. To formulate this result, we introduce  a parameter $r$, associated with a generating set  $\cG$, defined by 
\begin{equation}
    \label{eqbG}r=r(\cG)=\sup\left\{\frac1k \log |\cG_k|: k\in\bN\right\}.\end{equation}
    where $\cG_k=\{g\in\cG: |g|=k\}$.
With this definition in place, we can now state our second main result.

\begin{thmx}\label{thmmain2}
Let $X = X(\cG)$ be a coded shift with generating set $\cG$ such that $\cG$ uniquely represents $X_{\rm con}(\cG)$ and $\hcon(X)>\hres(X)$. If $h_{\rm top}(X)>r(\cG)$, then the Vere-Jones parameter $\kp$ is $\cG$-computable and the unique MME $\mu_{\rm max}$ of $X$ is   computable based on having oracle access to $\cG$, the language $\cL(X)$, as well as a rational $\varepsilon$ satisfying $0<\varepsilon<h_{\rm top}(X)-r(\cG)$.

\end{thmx}

We note that while it is easy to see that $\frac1k \log |\cG_k|<h_{\rm top}(X)$ for all $k\in\bN$, it is not known if $h_{\rm top}(X)>r(\cG)$ holds in general.

Our techniques can also be applied to establish the computability of an entropy maximizing measure on the concatenation set in the case $\hres(X)\geq\hcon(X)$. We continue to refer to $\kp(\cG)=\sum_{g\in \cG} |g| \exp(-|g|\hcon(X))$ as the Vere-Jones parameter of $\cG$. We say that a generating set $\cG$ has bounded growth rate if $b(\cG)=\sup\{|\cG_k|: k\in \bN\}<\infty$. Coded shifts with bounded growth rate, which were first considered in \cite{kucherenko_ergodic_2024},  contain several well-known classes of shift spaces, e.g. $\beta$-shifts and generalized gap shifts.

\begin{thmx}\label{thmmain3} Let $X=X(\cG)$ be a coded shift with generating set $\cG$ such that $\cG$ uniquely represents $\Xcon$. Suppose that there exists $\mu\in \cM_\sigma(X)$ with $\mu(\Xcon)=1$ and $h_\sigma(\mu)=\hcon(X)$. Then
\begin{itemize}
    \item[(i)] $\mu$ is the unique invariant entropy-maximizing measure putting full measure on $X_{\rm con}$.
    \item[(ii)] $\mu$ is $\cG$-Bernoulli with $p_g=\exp(-|g|\hcon(X))$ and $\kp=\sum_{g\in \cG} |g| \exp(-|g|\hcon(X))$.
    \item[(iii)] Suppose $\cG$ has bounded growth rate with an upper bound of $b(\cG)$ given by an oracle. Then the Vere-Jones parameter $\kp$ is $\cG$-computable. Moreover,  the measure $\mu$ is computable based on having oracle access to $\cG$.
\end{itemize} 
\end{thmx}

The first two statements in Theorem \ref{thmmain3} (uniqueness of the measure $\mu$ and its $\cG$-Bernoulli property) were proven in \cite{kucherenko_ergodic_2024}
under the assumption of $\hcon(X)\geq\hres(X)$. Here we provide the general result. But it is important to note that the main novelty of Theorem \ref{thmmain3} is the computability of $\mu$ based on the assumption of having oracle access to $\cG$ and $b(\cG)$.

Next, we apply our results to several well-known classes of shift spaces including $S$-gap shifts, generalized gap shifts, and $\beta$-shifts (see Section \ref{examples} for definitions and details):

\begin{corollary}
    Let $\bt>1$ be a non-integral real number and let $X_\bt$ be the $\bt$-shift together the standard generating set $\cG=\cG(\bt)$  defined in Section \ref{beta_generators}. Then  the unique MME $\mu_\beta$ of $X_\bt$ is  computable based on oracle access to $\cG$.
\end{corollary}
We stress that in the above corollary besides the oracle of $\cG$ the algorithm also relies on specific properties of $\beta$-shifts.
Next we consider generalized gap shifts which are given by a collection of $d$ subsets of the natural numbers $S_0,\dots,S_{d-1}$ and a permutation set $\Pi$. The case $d=1$ reduces to the  case of the $S$-gap shift. We refer to Section 6.2 and \cite{burr_computability_2024, kucherenko_ergodic_2024} for more details.
\begin{corollary}
    Let $X$ be a generalized gap shift given by sets  $S_0,\dots,S_{d-1}\subset \bN$ and a permutation set $\Pi.$ Then the unique MME $\mu_{\rm max}$ of $X$ and the Vere-Jones parameter $\kp$ associated with the standard generating set $\cG$  (see Section \ref{sec62}) are computable based on having oracle access to $S_1,\dots,S_{d-1},$ and $\Pi.$
\end{corollary}


We note that  the computability of the Vere-Jones parameter $\kp$ for $S$-gap shifts, multiple gap shifts, and $\beta$-shifts is immediate since the number of generators with the same length is either at most one (for the $\beta$-shift and $S$-gap shift) or the constant $\epsilon$ in Theorem \ref{thmmain2} can be derived from the input data.  Next we combine Theorem \ref{thmmain3} with results in \cite{kucherenko_ergodic_2024} to derive computability results for the measures of maximal entropy for coded shifts with non-unique measures of maximal entropy. Specifically, we consider the well-known Dyke shift and show that each of the two ergodic measures of maximal entropy is computable.
\begin{corollary}
Let $X$ be the standard Dyke shift. Then both ergodic measures of maximal entropy are computable.
 \end{corollary}

We recall that for $\hres(X)\geq \hcon(X)$ the measure of maximal entropy is, in general, not unique, see \cite{kucherenko_ergodic_2024}. Nevertheless, there are many examples of a coded shift $X(\cG)$ for which the inequality $\hres(X)\geq \hcon(X)$ holds but the MME is still unique. However, even under the uniqueness assumption of the MME and the computability of the Vere-Jones parameter, the MME is, in general, not computable.

\begin{thmx}\label{thmnoncomput}
There exist a coded shift  $X=X(\cG)$  with generating set $\cG$ such that $\cG$ uniquely represents $\Xcon$ with the following properties:
\begin{enumerate}
    \item[(i)]$\hres(X)>\hcon(X)$;
     \item[(ii)] $X$ has a unique MME $\mumax$;
     \item[(iii)] The generating set $\cG$ has bounded growth rate, i.e., $b(\cG)<\infty$;
     \item[(iv)]The Vere-Jones parameter $\kp$ is computable from having oracle access to $\cG$ and a rational upper bound of $b(\cG)$;
      \item[(v)]The unique MME $\mumax$ is not computable from oracles of $\cG, \cL(X)$ and $b(\cG)$.
\end{enumerate}
\end{thmx}

\subsection*{Outline of the paper} 
In Section \ref{definitions} we review relevant background material and definitions from computability theory and from ergodic theory for coded shift spaces. Section \ref{sec:computability2} is devoted to the proofs of Theorems \ref{thmmain} and \ref{thmmain2}.  Theorem \ref{thmmain3} is proven in Section \ref{sec:proof_non_computability}. In Section \ref{sec:parameter} we construct a rather general example for the non-computability of the Vere-Jones parameter. In Section \ref{examples}, we apply our main theorems to prominent classes of coded shifts, including S-gap shifts, generalized S-gap shifts, Beta-shifts and the Dyke shift.

\section{Background and Definitions}\label{definitions}

\subsection{Basics from computability theory}\label{sec:computability}
In this section we introduce some basic concepts in computability theory that we use in this paper.
To simplify the exposition we start with the computability of real numbers.

\begin{definition}\label{sec:def_computability}
Let  $x\in \bR$. An \emph{oracle} of $x$ is a function $\phi:\bN\to \bQ$ such that $\vert \phi(n)-x\vert < 2^{-n}$. Moreover, we say $x$ is computable
if there is a Turing Machine $T=T(n)$ 
which is an oracle of $x$.
\end{definition} 
For our purposes it is enough to think of Turing machines as an algorithm that takes a finite number of inputs and carries out a finite number of computations to produce an output. For more details on Turing machines and computable analysis we refer to \cite{weihrauch_computable_2000}.

It is straightforward to see that rational numbers, algebraic numbers, and some transcendental numbers such as ${ e}$ and $\pi$ are computable real numbers. However, since the collection of Turing machines is countable, most points in $\bR$ are not computable.

Next, we define computable functions on the real line.
\begin{definition}\label{defcompfuncR} Let $D\subset \bR$. A function $f:D\rightarrow\bR$ is {\em computable} if there is a Turing machine $T$ so that for any $x\in D$, any oracle $\phi$ for $x$ and any $n\in \bN, \,\, T(\phi,n)$
is a rational number so that $\vert T(\phi,n)-f(x)\vert<2^{-n}$.
\end{definition}
We observe that one of the inputs of the Turing machine $T$ in Definition \ref{defcompfuncR} is an oracle. Specifically, while the Turing machine $T$ in principle has access to an infinite amount of data, it must be able to decide when the approximation $\phi(m)$ of $x$  is sufficiently accurate to perform the computation of $f(x)$ to precision $2^{-n}$.
We further note that in Definition \ref{defcompfuncR}, the input points $x$ are not required to be computable. In fact, any set  $D\subset \bR$ can be the domain of a computable function $f$. We finally point out that computable functions must necessarily be continuous but the converse is not true in general.

Next, we extend the notion of computable points to more general spaces, called computable metric spaces.

\begin{definition}\label{defcompmetric}
    Let $(X,d)$ be a separable complete metric space and let $\cS=\{s_i: i \in \bN\}$ be dense in $X.$  We say that $(X,d,\cS)$ is a computable metric space if there exists a Turing machine $\phi\colon\NN^2\times\NN\to\QQ$ such that $|\phi(i,j,n)-d(s_i,s_j)|<2^{-n}$.  
    We call $\cS$ the set of ideal points of $X$. 
    is an oracle of $x\in X$ if $d(x,s_{\phi(n)})<2^{-n}.$ Moreover,
 $x\in X$ is computable if there exists a Turing machine $T=T(n)$ which is an oracle of $x$
\end{definition}

Roughly speaking, in the above definition, the Turing machine $\phi$ is able to approximate distances between ideal points at any given accuracy. We also say that $\phi$ uniformly computes $d$ on $\cS\times \cS$.  We can recover from Definition \ref{defcompmetric} the definition of computable real numbers (see Definition \ref{sec:def_computability}) by defining $X=\RR$ with the standard metric and selection $\cS$ to be a enumeration of $\QQ.$ 
Next, we extend the notion of computable functions to functions between computable metric spaces as follows.

\begin{definition}\label{def:computablefunction}
Let $(X,d_X,\cS_X)$ and $(Y,d_Y,\cS_Y)$ be computable metric spaces and  $\cS_Y=\{t_i: i\in \bN\}$.  Let $D\subset X$.  A function $f:D\rightarrow Y$ is {\em computable} if there is a Turing machine $T$ such that for any  $x\in D$ and any oracle $\phi$ of $x$, the output $T(\phi,n)$ is a natural number satisfying $d_Y(t_{T(\phi,n)},f(x))< 2^{-n}$.

\end{definition}Analogously to the case of real-valued functions, computable functions are continuous. 

Let $(X,d,\cS)$ be a computable metric space. We are interested in the space of all Borel probability measures on $X$ endowed with the weak$^\ast$ topology. We denote this space by $\cM(X).$ There is a natural way to define a computable metric space structure on $\cM(X)$, see \cite{hoyrup_computability_2009}. Namely, we consider probability measures supported on a finite subset of $\cS$ and having rational weights, i.e., measures of the form 
\[
\sum_{i\in I} q_i\delta_{s_i},
\]
where $I$ is a finite subset of $\NN$, $\delta_{s_i}$ is the Dirac measure on $s_i$  and $\{q_i\colon i\in I\}$ is a set of positive rationals that sum up to $1.$ Measures of this form constitute the set of ideal points in $\cM(X),$ which we denote by $\cS_{\cM(X)}.$ From now on we assume  that $X$ is bounded. Recall that the topological space $\cM(X)$ is metrizable. One metric which is compatible with the weak$^\ast$ topology is the Wasserstein-Kantorovich metric.

\begin{definition}\label{defWK}
    The Wasserstein-Kantorovich metric $W_1$ on $\cM(X)$ is defined by
    
    \[
    W_1(\mu, \nu)=\sup_{f\in\text{1-Lip}}\left|\int f d\mu -\int fd\nu\right|,
    \]
where $\text{1-Lip}$ denotes the set of Lipschitz functions $f\colon X\to\RR$ with Lipschitz constant $1.$
\end{definition}
The metric $W_1$ belongs to a family of metrics $W_p$ which are interesting in their own right. They are a natural class of metrics to work with in the context of optimal mass transport and have found broad applicability; see, e.g., \cite{panaretos}.
We note that if $X$ is bounded then $(\cM(X),\cS_{\cM(x)}, W_1)$ is a computable metric space \cite{hoyrup_computability_2009}. This allows us to define computable measures in $\cM(X)$.
\begin{definition}\label{defmui}
    Fix an enumeration $\{\mu_i: i\in\NN\}$ of the ideal measures $\cS_{\cM(X).}$ A measure $\mu\in\cM(X)$ is computable if there exists a Turing machine $\chi\colon\NN\to\NN$ such that $W_1(\mu,\mu_{\chi(n)})<2^{-n}.$
\end{definition}
We will use this definition to consider the computability of the measures of maximal entropy on coded shift spaces.
\subsection{Coded shifts}\label{coded shifts}
We continue to use the notation from Section 1.2.
For a  finite alphabet $\cA=\{0,\dots,d-1 \}$ we consider the product space $\Sigma_d=\cA^{\ZZ}$ which we call the full shift in $d$ symbols. Endowing $\Sigma_d$ with the Tychonoff product topology  makes it a compact metrizable topological space. For example, the (standard) metric given by
\begin{equation}\label{defmetX}
d(x,y)=  2^{-\min\{|k| \;:\;  x_k\neq y_k\}} 
\end{equation}
induces the Tychonoff product topology on $X$.
Let $\sigma\colon \Sigma_d\to \Sigma_d$ be the left shift map which is a homeomorphism.  A shift space (also called a subshift) is a closed shift invariant subset $X$ of $\Sigma_d$. A shift space $X$ is  coded  if there exists a countable set $\cG$ (called a generating set) of words (i.e. elements in $\cA^*=\bigcup_{n=1}^{\infty}\cA^n$) such that $X$ is the smallest shift space containing the set of all free concatenations of  elements of $\cG$. Precisely, $X$ is the closure of the concatenation set defined by 
\[
\Xcon=\{\dots g_{-1}g_{0}g_{1}\dots \colon g_{i}\in\cG \}.
\]
We recall the definition of the residual set $\Xres = X\backslash\Xcon$. We note that both the concatenation set and the residual set depend on the generating set $\cG$ and it is possible that different generating sets produce the same coded shift.
Coded shifts are quite general and include transitive sofic and almost sofic shifts \cite{lind_introduction_2021}, as well as many familiar examples such as $S$-gap shifts, $\beta$-shifts, and various generalizations including generalized gap shifts and intermediate $\beta$-shifts.
We define the {\bf language} of a shift space $X$ by 
\[
\cL(X)= \bigcup_{n\ge 0} \cL_n(X),
\]
where $\cL_0=\emptyset$ and for $n\geq 1$,
\[
\cL_n(X) = \{w_1\dots w_n\in \cA^n \colon \exists x\in X\,\, {\rm with }\,\, x_1=w_1, \dots,x_n = w_n \}.
\]
Given a word $w\in \cA^n$, we say $w$ has length $n$  which we denote  by $|w|.$ We denote by $|S|$ the cardinality of a set $S$.

Central to our settings are the concepts of topological and measure-theoretic entropy, which we briefly define in the special case of shift spaces. For more general definitions and properties see, for example, \cite{viana_foundations_2016}. The {\bf topological entropy} of $X$ is defined by
\[
\htop(X) = \lim_{n\to \infty}\frac{1}{n}\log |\cL_n(X)|.
\]
Notice that $|\cL_n(X)|\leq d^n$ and, as a consequence of the sub-additivity of $\log |\cL_n(X)|$, $|\cL_n(X)|\ge \exp(n\htop(X))$. 

For the measure-theoretic counterpart of the topological entropy we denote by $\cM_\sigma(X)$ the subset of $\cM(X)$ consisting of those measures which are invariant under the shift map; i.e., that satisfy $\mu\circ\sigma^{-1}=\mu.$ The {\bf measure-theoretic entropy} of  $\mu\in \cM_\sigma(X)$ is defined by

\[
h_{\mu}(\sigma) = \lim_{n\to\infty}-\frac{1}{n}\sum_{w\in \cL_n(X)}\mu([w])\log\mu([w]),
\]
where $[w]:=\{x\in X\colon x_i=w_i, 0\le i\le |w|-1 \}$ is the (standard) cylinder set of $w.$ The topological and measure-theoretic entropies are related by the following variational principle:
\[
\htop(X)=\sup\{h_\mu(\sigma)\colon \mu\in\cM_\sigma(X)\}.
\]
An invariant measure that achieves the supremum above is called a {\bf measure of maximal entropy} (or MME). If the MME is unque we denote it by $\mumax.$ It is well known that every shift space has an MME, but it may not be unique. In case of coded shifts, uniqueness of the MME was proven in \cite{kucherenko_ergodic_2024} in the case $\hcon(X)>\hres(X),$ where $\hcon(X)$ and $\hres(X)$ are defined as in equations \eqref{eq1} and \eqref{eq2}, respectively.
Furthermore, it is also shown that $\mumax$ is {\bf $\cG$-Bernoulli}; that is, there exist positive numbers $p_g, g\in \cG$ with

$$\sum_{g\in \cG}p_g =1\text{ and } c:=\sum_{g\in\cG}|g|\exp(-|g|\htop(X))<\infty,$$
such that
\[
\mumax\left( \llbracket g_0 \ldots g_k \rrbracket\right)=\frac{1}{c}p_{g_0} \cdots p_{g_k}.
\]
Here  $\llbracket g_0 \ldots g_k\rrbracket$ is the {\bf $\cG$-cylinder} defined by
\[
\llbracket g_0 \ldots g_k\rrbracket = \{x\in \Xcon\colon x=\dots g'_{-1}.g'_0g'_1\dots, g'_i\in\cG \text{ and } g'_j=g_j \text{ for } 0\le j\le k. \}
\]
A crucial property of $\cG$ needed for our results (and for the results in \cite{kucherenko_ergodic_2024}) is that $\cG$ {\bf uniquely represents} $X_{\rm con}$. The generating set $\cG$ uniquely represent $X_{\rm con}$ (or $\cG$ is uniquely decomposable in \cite{pavlov_entropy_2020}) if for every  $x\in\Xcon$ there is a unique sequence $(g_i)_{i\in \bZ}$ in $\cG$ such that $x=\cdots g_{-1}g_{0}g_{1}\cdots.$ Since  every coded shift has a generating set that uniquely represents $X_{\rm con}$ \cite{beal_unambiguously_2024}, we work under the standing assumption that our generating set $\cG$ uniquely represents $X_{\rm con}$. 

It is shown in \cite{kucherenko_ergodic_2024} that the uniqueness of the MME fails in general when $\hres(X)\ge \hcon(X)$. In fact, \cite{kucherenko_ergodic_2024} provides several examples, as well as a constructive method for producing coded shifts with multiple MMEs in the case $\hres(X)\ge \hcon(X).$ However, we will show next that if a measure in 
\[
\cM_{\rm con}(X) := \left\{\mu\in\cM_{\sg}(X) \colon \mu(\Xcon) = 1 \right\}
\]
has metric entropy equal to $\hcon(X)$, then this measure is unique among those in $\cM_{\rm con}(X)$ and is $\cG$-Bernoulli.

\begin{proposition}\label{A1}
    Suppose $\cG$ uniquely represents $\Xcon$ and that there exists $\mu\in\cM_{\rm con}(X)$ such that $h_{\sg}(\mu) = \hcon(X).$ Then $\mu$ is the unique entropy maximizing measure among those in $\cM_{\rm con}(X)$ and it is a $\cG$-Bernoulli measure with $$p_g = \exp(-|g|\hcon(X))\text{ and }c = \sum_{g\in\cG}|g|\exp(-|g|\hcon(X)).$$ 
\end{proposition}

\begin{proof}
    The statement is proven in \cite{kucherenko_ergodic_2024} under the assumption $\hcon(X)>h_{\rm res}(X)$. The proof in the general case is analogous with one exception. We briefly review the idea of the proof in the $\hcon(X)>h_{\rm res}(X)$ case. The authors consider the topological pressure $P_{\rm top}(\phi)$ of the potential $\phi=-h_{\rm top}(X)$. Clearly, $P_{\rm top}(\phi)=0$ and $\mu\in \cM_\sigma(X)$ is an equilibrium state of $\phi$ if and only if $\mu$ is an MME of $X$. The uniqueness of $\mu$ and its identification as the corresponding $\cG$-Bernoulli measure is then established by inducing on the set \[
    E=\{x=\cdots g_{y_{-2}}g_{y_{-1}}.g_{y_{0}}g_{y_{1}}g_{y_{2}}\cdots\},\]
    i.e., the set of points $x\in X_{\rm con}$ with a generator starting at $x_0$.
    One can bijectively relate equilibrium states of $\phi$  with equilibrium states of $\phi_E\colon E\to\RR$ defined by
    \[
    \phi_E(x) = \sum_{j=0}^{|g_{y_0}|-1}\phi\circ\sg^j(x) = -|g_{y_0}|\hcon(X),
    \]
    Using that $E$ can be identified with the two-sided full shift over the alphabet $\bN$ and applying the theory of equilibrium states for countable Markov shifts, e.g. \cite{MU}, one can then deduce the desired result. We refer to \cite{kucherenko_ergodic_2024} for the detailed proof. 
    
    The proof in the general case, that is, without assuming $\hcon(X) > h_{\rm res}(X)$, for a measure $\mu \in \cM_{\rm con}(X)$ satisfying $h_\sigma(\mu) = \hcon(X)$ proceeds along the same lines, with one notable modification. Namely, instead of considering $P_{\rm top}(\phi)$ one has to consider the concatenation pressure $P_{\rm con}(\phi)$ on $X_{\rm con}$ defined
    by 
    \[P_{\rm con}(\phi)=\sup\left\{h_{\sigma}(\nu)+\int \phi\, d\nu: \nu\in \cM_{\rm con}(X) \right\},\]
    where $\phi=
    -\hcon(X)$. The concatenation pressure was introduced in \cite{burr_computability_2022}.
    Clearly, the measure $\mu$ satisfies $P_{\rm con}(\phi)=h_{\sigma}(\mu)+\int \phi\, d\mu=0$. After inducing on $E$ the remaining arguments of the proof are then analogous to the corresponding arguments of the proof for the $\hcon(X)>h_{\rm res}(X)$ case.
\end{proof}

\subsection{Computability from oracles of $\cL(X)$ and $\cG$}

When trying to establish computability in the context of coded shifts, a natural approach is to approximate a coded shift $X$ generated by $\cG$ by the coded shift generated by all $g\in\cG$ up to a certain finite length and to ask whether such a finitely generated subshift gives a qualitatively meaningful approximation of $X$. For this purpose we define

\[
\cG_n = \left\{g\in\cG \colon |g|=n \right\},
\]
and
\[
\cG_{\leq N} = \bigcup_{n\leq N}\cG_n.
\]

The following definitions provide a notion of computability when the language and/or the generating set of a coded shift act as inputs.

\begin{definition}\label{def:oracles_of_G_and_L(X)}
    We say that $\phi\colon\NN\to \cA^*$ is an oracle of $\cL(X)$ if $\phi(\NN)=\cL(X),$ and $\phi$ is non-decreasing in word length, i.e., $|\phi(n+1)|\ge |\phi(n)|$.
Suppose $\cG$ is an infinite generating set. We say that $\psi\colon\NN\to \cA^*$ is an oracle of $\cG$ if $\psi(\NN)=\cG,$ and $\psi$ is non-decreasing in word length. If $\psi$ is an oracle of $\cG$ we also write $g_i=\psi(i)$.
If $\cG$ is a finite generating set we substitute the domain $\NN$ with $\{1,\dots,|\cG| \}.$
\end{definition}

\begin{definition} Suppose we are given oracles of $\cG=(g_i)_{i\in\bN}$ and $(p_i)_{i\in \bN}$ such that $\sum_{i=1}^\infty p_i=1$ and $c=\sum_{i=1}^\infty |g_i|p_i<\infty$. We say that  $c$ is $\cG$-computable if  there exists a Turing machine $T(n)$ which on input $n \in \mathbb{N}$, outputs a positive integer satisfying $$\left|\sum_{i\leq T(n)}|g_i| p_i-c\right|<2^{-n}.$$
In particular, given an oracle of the concatenation entropy $\hcon(X)$, the Vere-Jones parameter $\kp=\kp(\cG)=\sum_{g\in \cG} |g| \exp(-|g|\hcon(X))$ is $\cG$-computable if there exists a Turing machine $T(n)$ such that $$\left|\sum_{i\leq T(n)}|g_i| \exp(-|g_i|\hcon(X))-\kp\right|<2^{-n}.$$
\end{definition}

We observe that if $\cG$ is finite then the Vere-Jones parameter is always $\cG$-computable from a given oracle of $\cG$. This is because for a finite generating set $\cG=\{g_1,\dots g_\ell\}$, we have $X=X_{\rm con}$ is sofic and $h_{\rm con}(X)= h_{\rm top}(X)$ can be computed from the (finite) characteristic equation $\sum_{i=1}^\ell \exp(h_{\rm con}(X))^{-|g_i|}=1. $ If $\cG$ is infinite then additional information is needed to compute $h_{\rm con}(X)$ and $\kp$.   
For example, if $h_{\rm con}(X)>h_{\rm res}(X)$ then $h_{\rm con}(X)$ can be computed from having oracle access to $\cG$ and $\cL(X)$, see \cite{burr_computability_2022}. While for some coded shifts it is possible to compute an oracle of $\cL(X)$ from an oracle of $\cG$, e.g., for $S$-gap shifts, this is in general not the case, see \cite{burr_computability_2022} for a more detailed discussion of this topic.

A weaker notion of computability is that of {\bf lower semi-computability}, which is the existence of a Turing machine that produces a non-decreasing sequence that converges to the object of interest. Note that such a Turing machine does not provide an error estimate. There is a corresponding notion of {\bf upper semi-computability}, and computability is equivalent to being simultaneous upper and lower semi-computability. To obtain computability from simultaneous upper and lower semi-computability it is enough to obtain an error bound using the difference of the upper and lower approximations. In the case of the Vere-Jones parameter it is easy to obtain lower semi-computability by computing the partial sum of the Vere-Jones parameter for all generators up to length $n$ at precision $2^{-n}$.

\begin{proposition}
    Let $\cG$ be a generating set. Then the Vere-Jones parameter $\kp=\kp(\cG)$ is lower semi-computable based on having oracle access to $\cG$ and $h_{\rm con}(X)$.
\end{proposition}
\begin{proof}
Given $N\in\bN$ we define $\kp_N=\sum_{k=1}^{N}|\cG_k|k\exp(-kh_{\rm con}(X))$ and observe that $(\kp_N)_N$ is an increasing sequence converging to $\kp$. Next, we use the oracle of $h_{\rm con}(X)$ to compute for each $N$ a real number $t_N\leq \kp_N$ satisfying
$\kp_N-t_N<1/N$. It follows that $(\max\{t_1,\dots,t_N\})_N$ is a non-decreasing computable sequence converging to $\kp$ from below. 
    \end{proof}

\section{Proofs of Computability results for $\cG$-Bernoulli measures and the MME}\label{sec:computability2}
In this section we prove Theorems \ref{thmmain}, \ref{thmmain2}, and \ref{thmmain3}. We start with two auxiliary results which are the basis of the proof of Theorem \ref{thmmain}. Let $X$ be a shift space given by an oracle of $\cL(X)$. Let $\mu\in \cM(X)$. Note that we do not assume for now that $\mu$ is shift-invariant. To work with a general measure we need the notation for general cylinders. For $w\in\cL(X)$ and $r,s\in\bZ$ such that $s-r+1=|w|$ we let
$$[w]_r^s=\{x\in X: x_r\cdots x_s=w\}.$$ In particular, if $k\in \NN$ and $w\in \cL_{2k+1}$ we call $[w]_{-k}^k$ a centered cylinder. 
We say a Turing machine $\chi\colon\cL(X)\times\bZ\times \NN \to \QQ$ computes $\mu$ on cylinders if for all $w\in \cL(X)$, $k\in \bZ$, and $n\in\NN$
    \begin{equation}\label{computcyl}
    \left|\mu( [w]_k^{k+|w|-1}) - \chi(w,k,n) \right| < 2^{-n}.
    \end{equation}
   If such a Turing machine $\chi$ exists we also say that $\mu$ is \emph{computable on cylinders}.
\begin{remark}\label{remcent}
    Note that computability of a measure on centered cylinders is equivalent to it being computable on all cylinders. Indeed,  given $v\in \cL(X)$, the non-centered cylinder $[v]_r^s$ can be written as the finite disjoint union of those centered cylinders $[w^i]_{-N}^N$ where $w^i\in \cL_{2N+1}(X)$, $N=\max\{|r|,|s|\}$ satisfying $w^i_r\dots w^i_s=v$. Clearly the cylinders $[w^i]_{-N}^N$ in this disjoint union can be computationally identified from $v$ and a given oracle of $\cL(X)$. Thus we can compute $\mu([v]_r^s)$ at precision $2^{-n}$ by computing the corresponding $\mu([w^i]_{-N}^N)$ at precision $\frac{2^{-n}}{|\cL_{2N+1}(X)|}$ and adding the outputs of the corresponding approximations. 
\end{remark}
 We further note that if the measure $\mu$ is invariant then \eqref{computcyl} needs only to hold for standard cylinders $[w]$, i.e. the Turing machine $\chi=\chi(w,n)$ does not depend on $k$. 

We start by showing that a measure to be computable is equivalent to being computable on cylinders.

\begin{proposition}\label{mu_computable}
Let $X$ be a shift space over a finite alphabet which is given by an oracle of $\cL(X)$.
Then $\mu\in \cM(X)$ is computable on cylinders if and only if  $\mu$ is computable.
\end{proposition}
\begin{proof}We first prove $\Rightarrow$. Suppose that $\mu$ is computable on cylinders. 
We note that $X$ has a natural computable metric space structure $(X,d,\cS_X)$, see \cite{burr_computability_2022}. Here $d$ is the metric defined in \eqref{defmetX}.
    Our goal is to construct a Turing machine which, given input $n\in \mathbb{N}$, computes a positive integer $k$ for which the $k$-th ideal measure $\mu_k$ in the enumeration of the set $\cS_{\cM{(X)}}=\{\mu_k:k\in \bN\}$ (see Definition \ref{defmui}) satisfies
$W_1(\mu,\mu_k)<2^{-n}.$

Let $n\in\bN$ be given. For simplicity we write $\epsilon=2^{-n}$. 
We compute $N\in\bN$ such that 
${\rm diam}([w]_{-N}^N)=2^{-N}<\frac\epsilon3$ whenever $w\in \cL_{2N+1}(X)$. Let $\ell= |\cL_{2N+1}(X)|$. We fix an enumeration of
$\cL_{2N+1}(X)=\{w^1,\dots,w^\ell\},$ e.g., by ordering $\cL_{2N+1}(X)$ in the lexicographic order. Using that $\mu$ is computable on cylinders we may compute $q_1,\dots,q_l\in \bQ^+$ such that $\mu([w^i]_{-N}^N)\geq q_i$ and $|\mu([w^i]_{-N}^N)- q_i|<\frac{\epsilon}{3\ell}$ for all $i=1,\dots,\ell$.  This can be accomplished by first computing $\tilde{q}_i$ with $|\mu([w^i]_{-N}^N)-\tilde{q}_i|<\frac{\epsilon}{6\ell}$ and defining $q_i=\max\{0,\tilde{q}_i-\frac{\epsilon}{6\ell}\}$. It follows that $0\leq c_\ell=1-\sum_{i=1}^\ell q_i<\frac\epsilon3$. For each $i=1,\dots,\ell$ select  $x_i\in \cS_X\cap [w_i]_{-N}^N$. We define a measure $\nu$ by
\[
\nu=\sum_{i=1}^{\ell} q_i \delta_{x_i}  +c_\ell\delta_{x_\ell}.
\]
Clearly, $\nu$ is a probability measure and in particular belongs to $\cS_{\cM(X)}$. Thus, we can computationally select $k\in\bN$ such that $\nu=\mu_k$. It remains to be shown that $W_1(\mu,\nu)<\epsilon$ where $W_1(\mu,\nu)$ denotes the Wasserstein-Kantorovich metric of $\mu$ and $\nu$, see definition \ref{defWK}. 

Let $f\in 1$-${\rm Lip}(X)$. Thus, $|f(x)-f(y)|\leq 2^{-N}$ whenever $x,y\in w^i$ for some $i=1,\dots,\ell.$ 
Since $\int f d\mu-\int fd\nu=\int (f-c)\, d\mu-\int (f-c)\,d\nu$ for all $c\in \bR$, taking $c=\sup f $ we may assume without loss of generality that $|f(x)|\leq 1$ for all $x\in X$. It follows that
 \begin{align*}
    \left| \int_X f d\mu - \int_X f d\nu \right| &=\left| \int - c_\ell f\,d \delta_{x_\ell}+ \sum_{i=1}^\ell\left( \int_{[w^i]_{-N}^N} f d\mu - \int_{[w^i]_{-N}^N} q_i f d\delta_{x_i}\right) \right| \\[1em]
&= \bigg | \int -c_\ell f\,d \delta_{x_\ell}+ \sum_{i=1}^\ell \left( \int_{[w^i]_{-N}^N} f d\mu -   \int_{[w^i]_{-N}^N} f(x_i)\, d\mu\right) \\[1em]&\qquad +   \sum_{i=1}^\ell \left(\int_{[w^i]_{-N}^N} f(x_i)\, d\mu - \int_{[w^i]_{-N}^N} q_i f d\delta_{x_i}\right) \bigg | \\[1em]
    &\leq \left| \int -c_\ell f\,d \delta_{x_\ell}\right|+ \left|\sum_{i=1}^\ell \int_{[w^i]_{-N}^N} f - f(x_i)\, d\mu\right|   \\
    & \qquad+\left|\sum_{i=1}^\ell \left(\int_{[w^i]_{-N}^N} f(x_i)\, d\mu -  q_i f(x_i)\right) \right| \\[1em]
    &\leq  \int |c_\ell f|\,d \delta_{x_\ell} + \sum_{i=1}^\ell \int_{[w^i]_{-N}^N} \left|f - f(x_i)\right|\, d\mu 
   +  \sum_{i=1}^\ell  |f(x_i)|(\mu([w^i]_{-N}^N) -  q_i)  \\[1em]
&\leq  c_\ell + \sum_{i=1}^\ell \int_{[w^i]_{-N}^N} \frac{\epsilon}{3}\, d\mu +  \sum_{i=1}^\ell (\mu([w^i]_{-N}^N) -  q_i )\\[1em]  
    &<\frac{\epsilon}{3}+\frac{\epsilon}{3}+\frac{\epsilon}{3}=\epsilon.
 \end{align*}
This shows that $W_1(\mu,\nu)=W_1(\mu,\mu_k)<\epsilon$ which completes the proof of $\Rightarrow$.

To prove $\Leftarrow$, let us assume that $\mu$ is computable. Let $w\in \cL(X), k\in \bZ$ and $n\in \bN$. We aim to compute $\mu( [w]_k^{k+|w|-1})$ at precision $\epsilon=2^{-n}$. By Remark \ref{remcent} it is sufficient to consider  the case where $w\in \cL_{2N+1}(X)$ and $[w]_{-N}^N$ is a centered cylinder. Let $f=2^{-N}1_{[w]_{-N}^N}$, where $1_{[w]_{-N}^N}$ denotes the characteristic function of $[w]_{-N}^N$. Hence $f\in 1$-${\rm Lip}(X)$. It follows that $\mu([w]_{-N}^N)=2^N\int f d\mu$. Since $\mu$ is computable we can compute  a measure $\mu_k\in \cS_X$ such that $W_1(\mu,\mu_k)<2^{-(n+N)}$. Define $q=2^N\int fd\mu_k\in \bQ$ which is computable. It follows that 
\begin{align*}
\left|\mu([w]_{-N}^N)-q\right|&= \left| 2^N\int fd\mu-2^N\int f d\mu_k\right|\\
&= 2^N\left| \int fd\mu-\int f d\mu_k\right|\\
&\leq 2^N\cdot W_1(\mu,\mu_k)< 2^{-n}.
\end{align*}
This completes the proof for the computability of the measure of centered cylinders. 
\end{proof}
Next we establish the computability on cylinders for $\cG$-Bernoulli measures.

\begin{lemma}\label{lemcompmucyl}
Let $X = X(\cG)$ be a coded shift with generating set $\cG$ such that $\cG$ uniquely represents $X_{\rm con}(\cG)$, and let $(g_i)_{i\in \bN}$ be an oracle of $\cG$. Let $\mu$ be a $\cG$-Bernoulli measure given by positive real numbers $p_i$  and normalizing constant $c$,  see equation \eqref{measure}. Suppose $c$ is $\cG$-computable and $(p_i)_{i\in\bN}$ is uniformly computable.
 Then $\mu$ is computable on cylinders.

\end{lemma}

\begin{proof}First recall that since $\mu$ is invariant it is enough to establish that $\mu$ is computable on standard cylinders.
    For $w\in \cL(X)$ we write $[w]_{\text{con}}=[w]\cap X_{\rm con}$  and observe that $\mu([w])=\mu([w]_{\text{con}})$. 
    Given $g_{n_0},\dots ,g_{n_k}$ we write $w\sqsubset g_{n_0}\dots g_{n_k}$ if the word $w$ is a subword of $g_{n_0}\dots g_{n_k}$ such that $w$ has a nonempty overlap with all $g_{n_i}, i=0,\dots,k$. Hence, $k\leq |w|.$ Moreover, we denote by $n(w,g_{n_0}\ldots g_{n_k})$ the number of copies of $w$ in $g_{n_0}\dots g_{n_k}$ satisfying the nonempty overlapping property. Clearly, $w\sqsubset g_{n_0}\dots g_{n_k}$ implies $n(w,g_{n_0}\ldots g_{n_k})\geq 1$. Moreover, $n(w,g_{n_0}\ldots g_{n_k})\leq |g|+|w|-1$ for all $g\in \{g_{n_0},\dots ,g_{n_k}\}.$
    It is shown in \cite{kucherenko_ergodic_2024}  that $[w]_{\text{con}}$ can be decomposed into a  union of countably many sets as
    \begin{equation}\label{cyldecom}
    [w]_{\text{con}} = \bigcup_{\substack{w\sqsubset g_{n_0}\ldots g_{n_k}}}\bigcup_{m=1}^{n(w,g_{n_0}\ldots g_{n_k})}\sigma^{\ell_m}(\llbracket g_{n_0}\ldots g_{n_k} \rrbracket),
    \end{equation}
where the integers $\ell_m$ are chosen so that the copies of $w$ with the above mentioned overlapping property match with the first coordinates of $\sigma^{\ell_m}(\llbracket g_{n_0}\ldots g_{n_k}\rrbracket)$. It follows from the overlapping property and the fact that $\cG$ uniquely represents $X_{\rm con}(\cG)$ that the sets appearing in the union \eqref{cyldecom} are pairwise-disjoint.
Therefore,
\begin{equation}\label{eqmucyl}
\mu ([w]_{\text{con}}) = \sum_{\substack{w\sqsubset g_{n_0}\ldots g_{n_k}}} n(w,g_{n_0}\ldots g_{n_k}) \mu(\llbracket g_{n_0}\ldots g_{n_k} \rrbracket).
\end{equation}
We define ${\rm ml}(g_{n_0},\dots ,g_{n_k})=\max\{|g_{n_0}|,\dots, |g_{n_k}|\}$. In the remainder of the proof we show that to approximate $\mu([w])$ it is enough to consider terms  in \eqref{eqmucyl} for which ${\rm ml}(g_{n_0},\dots ,g_{n_k})$ is uniformly bounded by a sufficiently large constant. It follows from the definition of ${\rm ml}(g_{n_0},\dots ,g_{n_k})$ that
\begin{equation}\label{eqgut1}
n(w,g_{n_0}\ldots g_{n_k})\leq  {\rm ml}(g_{n_0},\dots ,g_{n_k})+|w|-1.
\end{equation}
Fix $n \in \mathbb{N}$. We now describe a procedure for computing $\chi(w,n)$. 
Since $c=\sum_{i=1}^\infty |g_i|p_i$ is $\cG$-computable, we can  compute $N_0\in \bN$ such that $\sum_{|g|>N_0} |g|p_g<2^{-n-2}$. Then for $N=\max\{N_0,|w|\}$ we have
\begin{equation}\label{eqgut3}
\sum_{|g|>N} (|g|+|w|-1) \mu(\llbracket g \rrbracket) =\sum_{|g|>N} (|g|+|w|-1) \frac1c p_g \leq  \sum_{|g|>N} (|g|+|w|-1) p_g <\sum_{|g|>N} 2|g| p_g < 2^{-n-1}.
\end{equation}
Let us consider a fixed $g\in \cG$.
It follows from the $\cG$-cylinder definition and \eqref{eqgut1} that
\begin{equation}\label{eqgut2}
\sum_{\substack{\ast}} n(w,g_{n_0}\ldots g_{n_k}) \mu(\llbracket g_{n_0}\ldots g_{n_k} \rrbracket) \leq (|g|+|w|-1) \frac1c p_g,
\end{equation}
where the sum on the left-hand side (condition $\ast$) runs over all $g_{n_0}\ldots g_{n_k}$ with $w\sqsubset g_{n_0}\ldots g_{n_k}$, $g\in \{g_{n_0},\ldots ,g_{n_k}\}$ and 
${\rm ml}(g_{n_0}\ldots g_{n_k})=|g|$. Thus, by \eqref{eqgut3} and \eqref{eqgut2}, 
\begin{equation}\label{eqasd1}
    \sum_{\substack{w\sqsubset g_{n_0}\ldots g_{n_k}\\ {\rm ml}(g_{n_0}\ldots g_{n_k})>N}} n(w,g_{n_0}\ldots g_{n_k}) \mu(\llbracket g_{n_0}\ldots g_{n_k} \rrbracket) <  2^{-n-1}.
\end{equation}
By combining \eqref{eqmucyl}  and \eqref{eqasd1} we conclude that to complete the proof it suffices to show the computability of  
\begin{equation}\label{qwe}
\sum_{\substack{w\sqsubset g_{n_0}\ldots g_{n_k}\\ {\rm ml}(g_{n_0}\ldots g_{n_k})\leq N}} n(w,g_{n_0}\ldots g_{n_k}) \mu(\llbracket g_{n_0}\ldots g_{n_k} \rrbracket)
\end{equation}
at precision $2^{-(n+1)}$. But this is possible since the finitely many terms of $g_{n_0}\ldots g_{n_k}$ in \eqref{qwe} can be computationally identified by a simple search algorithm and each of the real numbers $n(w,g_{n_0}\ldots g_{n_k}) \mu(\llbracket g_{n_0}\ldots g_{n_k} \rrbracket)$ can be computed at any precision from the $\cG$-computability of $c$ and the uniform computability of the $(p_i)_i$. We note that here it is sufficient to compute each of the numbers $n(w,g_{n_0}\ldots g_{n_k}) \mu(\llbracket g_{n_0}\ldots g_{n_k} \rrbracket)$ at precision $\ell^{-1}\cdot 2^{-(n+1)}$ where $\ell$ is is the number of terms occurring in the  sum \eqref{qwe}.

Finally, we note that the computational steps used to output $\chi(w,n)$ can be performed by a single Turing machine $\chi$ independently of $w$ and $n$ which completes the  proof.
\end{proof}

We note that in Lemma \ref {lemcompmucyl} we do indeed require $(p_i)_i$ to be uniformly computable. This is because in order to compute the sum in \eqref{qwe}
with increasing $n$ we need an increasing number of $p_i$ approximations. These computations are required to be made by one Turing machine which becomes part of the Turing machine $\chi$. If the $p_i$ were merely computable then $\chi$ would need to include infinitely many Turing machines (one to compute each $p_i$) which is impossible since a Turing machine is given by a finite code.

We now present the proofs of Theorems \ref{thmmain}, \ref{thmmain2} and \ref{thmmain3}, and Corollary \ref{cormain11}.
\subsection{Proofs of Theorem \ref{thmmain} and Corollary \ref{cormain11}}

\begin{proof}[Proof of Theorem \ref{thmmain}]
    The proof follows directly from Proposition \ref{mu_computable} and Lemma \ref{lemcompmucyl}.
    \end{proof}

\begin{proof}[Proof of Corollary \ref{cormain11}]By hypotheses, $h_{\rm con}(X)>h_{\rm res}(X)$. It is shown in \cite{kucherenko_ergodic_2024} that in this case the (unique) measure of maximal entropy is $\cG$-Bernoulli with $p_g=\exp(-|g|\htop(X))$ and $c=\kp=\sum_{g\in \cG} |g| \exp(-|g|\htop(X))$. 
Since $\kp$ is assumed to be $\cG$-computable, by Theorem \ref{thmmain} it suffices to show that $(p_g)_g$ is uniformly computable.
But this follows from the fact that $\htop(X)$ is computable based on oracle access to $\cG$ and $\cL(X)$ which was proven in \cite{burr_computability_2022}.
\end{proof}



\subsection{Proof of Theorem \ref{thmmain2}}

The proof of Theorem \ref{thmmain2} is a direct consequence of Corollary \ref{cormain11}, the fact that  $h_{\rm top}(X)=\hcon(X)$ is computable based on oracle access of $\cG$ and $\cL(X)$ (see \cite{burr_computability_2022}), and the following Lemma which provides sufficient conditions for the $\cG$-computability of the Vere-Jones parameter. We recall the definition of the parameter $r(\cG)$ of a generating set $\cG$ in equation \eqref{eqbG}.

\begin{lemma}\label{c_computable} Let $X=X(\cG)$ be a coded shift with generating set $\cG$ such that $\cG$ uniquely represents $\Xcon$  and $\hcon(X)-r(\cG)>\varepsilon>0$. Then the Vere-Jones parameter $\kp$ is $\cG$-computable from having oracle access to $\cG$, $\hcon(X)$, and $\varepsilon$. 

\end{lemma}

\begin{proof} We write $b=r(\cG)$. Since the oracle access  to $\hcon(X)$ and $\varepsilon$ is given, we can compute $q,\delta\in\bQ^+$ such that $b<q-\delta<q<q+\delta<\hcon(X)$. For $N\in \bN$ and $t\in  [q,\infty)$ consider the functions
\begin{equation}\label{eqft}
    f(t) = \sum_{g\in \cG} |g|\exp(-|g|t) =\sum_{k =1}^{\infty} |\cG_{k}| k \exp(-k t),
    \end{equation}
    and 
\begin{equation}\label{eqftN}
    f_N(t) =\sum_{k =1}^{N} |\cG_{k}| k \exp(-k t).
    \end{equation}
    Moreover, let $R_N(t)=f(t)-f_N(t)$ denote the remainder of $f$. It follows that $\kp=f(\hcon(X))$.  Since $|\cG_k|\leq e^{kb}$ for all $k\in\bN$ we conclude that
    \begin{equation}\label{eqR(t)}
     R_N(t)=\sum_{k=N+1}^\infty |\cG_k|k \exp(-k t)
     \leq \sum_{k=N+1}^\infty k\exp(-k(t-b))
      \leq \sum_{k=N+1}^\infty k\exp(-k\delta),
      \end{equation}
 which implies that $f$ is analytic on $(q,\infty)$. Performing  similar estimates on the derivative of $f$ 
 we deduce that
 \begin{equation}\label{eqlipf}
     |f'(t)|\leq \sum_{k=1}^\infty k^2 \exp(-k\delta) <\infty
 \end{equation}
 for all $t\in (q,\infty)$. Since $\delta$ was already explicitly computed, we may compute a rational number $L$ which is an upper bound of 
 $\sum_{k=1}^\infty k^2 \exp(-k\delta)$.
The Mean Value theorem implies that $L$ is a Lipschitz constant of $f\vert_{[q,\infty)}$.

Let now $n\in \bN$ be given. We desire to compute $\kp$ at precision $2^{-n}$. Recall the assumption that $\hcon(X)$ is given by an oracle. Thus, we may compute $h\in (q,\infty)\cap \QQ$ with $|h-\hcon(X)|< \frac{2^{-n-1}}{L}$. Moreover, by definition of $f(h)$ (also using equation \eqref{eqR(t)}), we can compute $r\in \bQ$ such that $|r-f(h)|\leq  2^{-n-1}$. Thus
\[|r-\kp|\leq |r-f(h)|+|f(h)-f(\hcon(X))|< 2^{-n-1}+L \frac{2^{-n-1}}{L}=2^{-n}. \]
Finally, we observe that in the computation of $r$ one can keep track of how many generators of the oracle of $\cG$ are required which completes the proof.
\end{proof}
A quantity closely related to $r(\cG)$ is the entropy of $\cG$ defined by
\[
h(\cG)=\limsup_{k\to\infty} \frac1k \log |\cG_k|.
\]
Obviously, $h(\cG)\leq r(\cG)$. Since $\cG$ is assumed to be infinite and since $ \frac1k \log |\cG_k| < \hcon(X)$ for all $k\in\bN$, it follows that $h(\cG)<\hcon(X)$ if and only if $r(\cG)<\hcon(X)$. Thus, we could have also considered  the criteria $h(\cG)<\hcon(X)$ in Lemma \ref{c_computable}. To the best of our knowledge it is not known if $r(\cG)=h(\cG)=\hcon(X)$ actually can occur. We stress that knowing that $r(\cG)<\hcon(X)$ does not automatically provide algorithmic access to  the positive gap $\varepsilon$ that is required in the proof of Lemma \ref{c_computable}.

\subsection{Proof of Theorem \ref{thmmain3}}
\begin{proof}[Proof of Theorem \ref{thmmain3}]Parts (i) and (ii) of Theorem \ref{thmmain3} are proven in
Proposition \ref{A1}. To prove part (iii) it suffices to show that the $\cG$-Bernoulli measure $\mu$ with $p_g=\exp(-|g|h_{\rm con}(X))$ and normalization constant $c=\kp=\sum_{g\in\cG}|g|\exp(-|g|h_{\rm con}(X))$ is computable. Let $L$ be an upper bound of $b(\cG)$ given by an oracle and let $N\in \bN$. By slightly increasing $L$ we may assume $L\in \QQ^+$ is a given upper bound of $b(\cG)$. We consider functions 
\begin{equation}
f(\lambda)=\sum_{g\in\cG} \lambda^{-|g|}=\sum_{k=1}^\infty |\cG_k|\lambda^{-k},
\end{equation}
and
\begin{equation}
f_N(\lambda)=\sum_{k=1}^N |\cG_k|\lambda^{-k}.
\end{equation}
Let $R_N(\lambda)=f(\lambda)-f_N(\lambda)=\sum_{k=N+1}^\infty |\cG_k|\lambda^{-k}$.
Hence $|R_N(\lambda)|\leq L\sum_{k=N+1}^\infty \lambda^{-k}$ which shows that $f$ is an analytic strictly decreasing function on $(1,\infty)$.
 Moreover, $f(\lambda)$ is computable on $(1,\infty)$ since $f_N(\lambda)$ is computable and $R_N$ can be made arbitrary small on any interval $[a,b)\subset (1,\infty)$. Here we also use the oracle access to $\cG$. By applying a simple interval division algorithm we conclude that the  unique solution $\lambda^\ast$ of the characteristic equation $f(\lambda)=1$ is computable. It is shown in \cite[Lemma 1]{kucherenko_ergodic_2024} that $\hcon(X)=\log \lambda^\ast$ which implies that $\hcon(X)$ is computable. This shows that $(p_g)_g=(\exp(-|g|h_{\rm con}(X)))_g$ is uniformly computable. 
 
 Next, we apply Lemma \ref{c_computable} to establish the $\cG$-computability of $\kp$. In order to establish the computability of $\epsilon>0$ with $\hcon(X)-r(\cG)>\epsilon$ we make the following observations: Since $\hcon(X)$ is computable we may compute $N\in \bN$ such that 
\[\max\left\{\frac{1}{k}\log |\cG_k|: k>N\right\}\leq \max\left\{\frac{1}{k}\log L: k>N\right\}<\frac12\hcon(X).\]
Note that $\frac1k\log |\cG_k|<\hcon(X)$ for all $k\in\bN$.  Therefore, by computing $\hcon(X)$ at large enough precision we can compute $\epsilon\in\bQ^+$ satisfying $\max\{\frac1k\log |\cG_k|: k=1,\dots, N\}<\epsilon <\hcon(X)$. By increasing $\epsilon$ if needed we can assure that $\epsilon>\frac12\hcon(X).$ We conclude that $\hcon(X)-r(\cG)>\epsilon$. Hence, by Lemma \ref{c_computable}, $\kp$ is $\cG$-computable. Finally,  applying Theorem \ref{thmmain} shows that $\mu$ is computable.
\end{proof}

\section{Proof of Theorem \ref{thmnoncomput}}\label{sec:proof_non_computability}

\begin{proof}[Proof of Theorem \ref{thmnoncomput}]
    Let $X^{(1)}$ and $X^{(2)}$ be transitive shifts with unique MMEs on the alphabets $\{3,4\}$ and $\{5,6\},$ respectively.  We further assume that $\htop(X^{(1)})>\htop(X^{(2)}).$ Let $0<\varepsilon<\htop(X^{(2)}).$ We choose a set $S\subset\NN$  such that $\lambda_*$ defined by
    \begin{equation}\label{eqsdfg}\sum_{n\in S}3\lambda_*^{-(n+1)}=1\end{equation}
    satisfies $\lambda_*<\exp(\varepsilon)$.
    Let $X^{(0)}$ be the $S$-gap shift defined by $S$, that is, $X^{(0)}$ is the coded shift with generating set
    \[\cG^{(0)} = \{0^s1\colon s\in S \}.\]
    We note that $\htop(X^{(0)}) < \log\lambda_* <\varepsilon.$ 
We write $\cG^{(0)}=\{g_i:  |g_{i}|<|g_{i+1}|\}.$

    Next we construct a coded shift $X$ over the alphabet $\{0,1,2,3,4,5,6\}$ with the desired properties.  Given $t\in\{ 1,2\}$ and $n\in\NN$ we select generators $g_n^{(t)}= 2w_nv2$ satisfying 
        \begin{enumerate}
            \item $w_n$ is a word containing all words in $\cL_n(X^{(t)})$
            \item $w_nv\in\cL(X^{(t)}),$ and
            \item $|g_n^{(t)}|=|g_{i_n}|$ for some increasing sequence $(i_n)_{n\in\NN}.$            
        \end{enumerate}
Now let $\cG = \cG^{(0)} \cup \cG^{(1)} \cup \cG^{(2)},$ where $\cG^{(t)}=\{g_n^{(t)}:n \in\NN\}$ for $t\in\{1,2\}.$ This defines our coded shift $X=X(\cG).$ Since for each $s\in S$ there are at most three $g\in\cG$ with $|g|=s$ while there are no $g\in\cG$ with $|g|\notin S$, it follows from \eqref{eqsdfg} and Lemma 1 in \cite{kucherenko_ergodic_2024} that
    \[\hcon(X)<\log\lambda_*<\varepsilon. \]
    Now, since every $x\in \Xcon$ contains at least one digit among $\{0,1,2\}$ and while sequences in $X^{(1)}\cup X^{(2)}$ do not, it follows that $X^{(1)}\cup X^{(2)}\subset \Xres$. Define

    \[
    \Xlim =\bigcap_{k\ge 1}\bigcup_{g\in\cG}\{x\in X\colon x[-k,k] \text{ is a subword of } g \}.
    \]
 Let $\mu\in \cM_\sigma(X)$ with $\mu(\Xres)=1.$
    It is shown in \cite{pavlov_entropy_2020} that $\mu(\Xcon\cup \Xlim)=1,$ so together with $\mu(\Xres)=1$ we obtain $\mu(\Xlim)=1.$ We observe that $\Xlim = \{0^\infty\}\cup X^{(1)}\cup X^{(2)},$ which  implies that $\mu(\{0^\infty\}\cup X^{(1)}\cup X^{(2)})=1$.
    Therefore, the variational principle implies $\htop(X)=\htop(X^{(1)}).$ Since $\htop(X^{(1)})>\htop(X^{(2)})$ it follows that $X$ has a unique MME $\mu_1$ which coincides with the unique MME of $X^{(1)}.$

    To show that $\mu_1$ is not computable from oracles of $\cL(X)$ and  $\cG,$ consider arbitrary $N,K\in\NN$ and define 
    \[
    \cL_{\leq N}(X) = \bigcup_{n\leq N}\cL_n(X),
    \]
    and
    \[
    \cG_{\leq K}(X) = \bigcup_{n\leq K}\cG_n(X).
    \]

    Let $Y$ be the coded shift generated by $\hat{\cG}=\cG^{(0)} \cup \cG^{(2)} \cup \cG_{\leq M},$ where $M=\max\{i_N,K\}.$ Then $\cL_{\leq N}(Y)=\cL_{\leq N}(X)$ and $\hat{\cG}_{\leq K}=\cG_{\leq K}.$ Note that $X^{(1)}\cap Y = \emptyset.$ Therefore, for any invariant probability measure $\mu$ on $Y$ with $\mu(Y_{\text{res}})=1$ we must have $\mu(\{0^\infty\}\cup X^{(2)})=1.$ It follows that $\htop(Y) = \htop(X^{(2)}).$ Furthermore, since $Y_{\text{con}}\subseteq \Xcon,$  it follows  that $\hcon(Y)\leq \hcon(X)<\varepsilon.$ We conclude that the MME of $Y$ is the MME of $X^{(2)},$ which we denote by $\mu_2.$ Since the MME $\mu_2$ of $Y=Y(N,K)$ is independent of $N$ and $K$ and $W_1(\mu_1,\mu_2) >0,$ we conclude that $\mu_1$ is not computable from the language and generators of $X.$ 

    To complete the proof, we observe that $\cG$ has bounded growth with $b(\cG)\leq 3$.
    Thus, it follows from Theorem \ref{thmmain3}\, (iii) that the Vere-Jones parameter $\kp$  is computable from having oracle access to $\cG$ and an upper bound of $b(\cG)$.
\end{proof}

\section{Computability of the Vere-Jones parameter}\label{sec:parameter}  
In this section, we consider the case $\hcon(X)>\hres(X)$ and show that, in general, the Vere-Jones parameter is not computable solely based on having oracle access to the generating set $\cG$ and the language $\cL(X)$ of a coded shift $X=X(\cG)$. This is surprising since having access to these oracles already implies the computability of the entropy of $X$ (see \cite{burr_computability_2022}), which, together with the generators, is the input in the formula of the Vere-Jones parameter.

\begin{example}\label{ex:Non-computability_VJP}
 There exists a coded shift $X=X(\cG)$ such that for all $N,M\in \bN$ there is a coded shift $X'=X(\cG')$ satisfying 
 \begin{itemize}
   \item $\hcon(X)>\hres(X)$ and $\hcon(X')>\hres(X')$;
   \item $\cL_N(X)=\cL_N(X')$ and the first $M$ elements of $\cG$ and $\cG'$ are the same;
   \item $|\kappa(\cG')-\kappa(\cG)|\ge \frac{1}{2}$.
 \end{itemize}
 In particular, the Vere-Jones parameter $\kappa$ of $X$ is not computable based on the oracle access to the generating set and the language of $X$.
 \end{example}
 \begin{proof}
   Let $\cA=\{0,1,2\}$. Consider a coded shift $X$ given by a generating set $\cG$ which contains no words of odd length and exactly two words of each even length. Precisely, we define
   \begin{equation}\label{eq:Def_G}
     \cG=\{01^{2n-1},\, 02^{2n-1}:n\in\bN \}.
   \end{equation} 
 It is easy to see that $\cG$ uniquely represents the concatenation set $\Xcon(\cG)$ and that the point-mass measures of $\bar{1}$ and $\bar{2}$ are the only two ergodic invariant measures which assign full measure to the residual set. Hence, $\hres(X)=0$, which immediately implies that $\hcon(X)>\hres(X)$ since any coded shift must have positive topological entropy. To compute $h_{\rm top}(X)$ we need to solve the equation 
 $\sum_{n=1}^{\infty}c(n)x^n=1$,
 where $c(2n)=2$ and $c(2n+1)=0$. We write
 $$\sum_{n=1}^{\infty}c(n)x^n=\sum_{n=1}^{\infty}2x^{2n}=\frac{2x^2}{1-x^2}$$
 and find $\lambda=3^{-\frac12}$ to be the positive solution to $\frac{2}{1-x^2}-2=1$. We conclude that $h_{\rm top}(X)=\frac{1}{2}\log 3$. We now can evaluate the Vere-Jones parameter of $X$:
 $$\kappa=\sum_{n=1}^{\infty}2\cdot 2n(3^{-\frac12})^{2n}=x\frac{d}{dx}\left(\frac{2x^2}{1-x^2}\right)|_{x=3^{-\frac12}}=3.$$
 By possibly enlarging $N$ and $M$ we may assume that $N$ is odd and $N=M$. Denote by $\cG^*$ the set of all finite concatenations of the elements of $\cG$. Consider the set of words
  \begin{equation}\label{eq:Def_G_N}
    \cG(N)=\{01^{N-1}w02^{N-1}:w\in\cG^*\}
  \end{equation}
  and let $X_N$ be the coded shift generated by $\cG\cup\cG(N)$. Note that $\cL_N(X_N)=\cL_N(X)$ and the first $N$ elements of the generating sets of $X$ and $X_N$ are the same ($X$ and $X_N$ have the same generating words up to length $2N$, then additional words of length greater or equal to $2(N+1)$ appear in the generating set of $X_N$). Since $X$ is a subshift of $X_N$ we have $h_{\rm top}(X)\le h_{\rm top}(X_N)$. Clearly $\cG\cup\cG_N$ uniquely represents the concatenation set of $X_N$ and the residual entropy of $X_N$ is zero. 
  Denote $c_N(n)= {\rm card} \{g\in\cG\cup\cG(N): |g|=n\}$. Then $c_N(n)=0$ whenever $n$ is odd, $c_N(2n)=2$ when $n\le N$, and $c_N(2n)=2+{\rm card}\{w\in\cG^*:|w|=2(n-N)\}$ for $n>N$. Using a simple induction argument we compute that $c_N(2n)=2+2\cdot 3^{n-N-1}$ for $n>N$. Indeed, we split the elements of $\cG_{2n}(N)\setminus\cG_{2n}=\{01^{N-1}w02^{N-1}:w\in\cG^*, |w|=2(n-N)\}$ into groups depending on the length of the first element from $\cG^*$ which appears in $w$. Observe that there are $2(c_N(2n-2k)-2)$ words of the form  $01^{N-1}w02^{N-1}$ with $w=g_1\dots g_m$ where $g_i\in\cG$ and $|g_1|=2k$. By the inductive hypothesis $c_N(2n-2k)-2=2\cdot 3^{n-k-N-1}$ for $k=1,...,(n-N)-1$ and there are 2 elements of $\cG_{2n}(N)$ for which $k=n-N$, so that
  \begin{equation*}
    c_N(2n)=2+\sum_{k=1}^{n-N-1}2\cdot 3^{n-k-N-1}+2=2+2\cdot 3^{n-N-1},
  \end{equation*}
  as claimed.

  We compute $h_{\rm top}(X_N)=-\log \lambda_N$ where $\lambda_N$ is the solution of the equation 
  \begin{equation}\label{eq:Entropy_X_N}
    \sum_{n\ge 1}2\lambda_N^{2n}+2\lambda_N^{2N}\sum_{n\ge 1}3^{n-1}\lambda_N^{2n}=1
  \end{equation}
  Since the topological entropy of $X$ is computable \cite{kucherenko_ergodic_2024}, $\lambda_N\to\lambda=\frac{1}{\sqrt{3}}$ from below. This can be also shown directly without invoking \cite{kucherenko_ergodic_2024}. Hence, we can find $N_0$ such that for $N\ge N_0$ we have
  \begin{equation}\label{eq:Lambda_N_conv}
    \frac{4\lambda_N^2}{(1-\lambda_N^2)^2}+\frac{1}{2(1-\lambda_N^2)}>3.5.
  \end{equation}
   To obtain the above inequality we used continuity of the right-hand side as a function of $\lambda_N$ and the fact that $\lim_{\lambda\to\frac{1}{\sqrt{3}}}\left(\frac{4\lambda^2}{(1-\lambda^2)^2}+\frac{1}{2(1-\lambda^2)}\right)=3.75$.  Let $\cG'=\cG\cup\cG(N_0)$, $X'=X_{N_0}=X(\cG')$, and $\lambda'=h_{\rm top}(X')$.
Consider the function $$f(x)=\sum_{n\ge 1}2x^{2n}+2x^{2N_0}\sum_{n\ge 1}3^{n-1}x^{2n}=\frac{2x^2}{1-x^2}+2 x^{2N_0}\left(\frac{x^2}{1-3x^2}\right).$$
  Then $f'(x)>\frac{4x}{(1-x^2)^2}+x^{2N_0}\frac{4x}{(1-3x^2)^2}$ whenever $x\in \left(0,\frac{1}{\sqrt{3}}\right)$. 
  Since by (\ref{eq:Entropy_X_N}) $f(\lambda')=1$,  a straight-forward calculation yields  $\frac{(\lambda')^{2N_0+2}}{(1-3(\lambda)'^2)^2}=\frac{1}{2(1-(\lambda')^2)}$.
Hence, 
  $$\kappa(\cG')=\lambda'f'(\lambda')>\frac{4(\lambda')^2}{(1-(\lambda')^2)^2}+\frac{1}{2(1-(\lambda')^2)}>\kappa(\cG)+\frac12.$$
 \end{proof}

Finally, we present an algorithm that establishes the computability of the topological entropy of the coded shift $X=X(\cG)$ solely based on  oracle access to $\cG$ and the computability of the Vere-Jones parameter.

\begin{proposition}\label{prop52}
Let $X=X(\cG)$ be a coded shift such that $\cG$ uniquely represents $X_{\rm con}$ and $h_{\rm con}(X)>h_{\rm res}(X)$. Suppose that the Vere-Jones parameter $\kp$ is $\cG$-computable based on having oracle access to $\cG$. Then $h_{\rm top}(X)$ is computable based on having oracle access to $\cG$.
\end{proposition}
\begin{proof}
Suppose $\cG=\{g_i: |g_i|\leq |g_{i+1}\}$ is given by an oracle. It follows from Propositions 24 and 29 in \cite{burr_computability_2022} that $h_{\rm top}(X)$ is lower semi-computable based on having oracle access to $\cG$. Here one also uses the well-know fact that the coded shift $X_m$ given by the first $m$ generators $\cG_m=\{g_1,\dots,g_m\}$ of $\cG$ is sofic and that a Turing machine can compute the associated finite directed labeled graph from $\cG_m$.

 We will show that $h_{\rm top}(X)$ is upper semi-computable by providing a simple search algorithm. 
Consider functions $$f(t)=\sum_{i=1}^\infty |g_i|\exp(-|g_i|t)\text{ and }f_\ell(t)=\sum_{i=1}^\ell |g_i|\exp(-|g_i|t),$$ where $\ell\in \bN$. Then $f(h_{\rm top}(X))=\kp$, and $f$ is a finite strictly decreasing function on $[\htop(X),\infty)$.
Let $T=T(n)=q_n$ be a Turing machine, which on input $n$ outputs a rational number $q_n$ such that $|\kp -q_n|<2^{-n}$. By subtracting $2^{-n}$ from $q_n$ and taking maxima we may assume that $(q_n)_n$ is non-decreasing and converges to $\kp$ from below.  

We now define a recursive algorithm that computes $\htop(X)$ from above. For $n=1$ we define $h_1$ to be a rational number greater or equal than $\log d$. Suppose  that $h_1,\dots,h_n$ have been computed. Since $\kappa$ is $\cG$-computable, we can compute the number $m=m(n)$ of the first $m$-generators in $\cG$ that $T$ queries to compute $\kp$ at precision $2^{-n}$ from below. Partition $[0,h_n]$ into $2^n$ equidistant rational numbers $0=t_1,\dots,t_{2^n}=h_n$. For each $t_i$  compute $f_m(t_i)$ at precision $2^{-n}$  denoted by $p_i$. Finally, define $h_{n+1}$ as the minimum of the set 
$\{t_i: p_i\leq q_n+2^{-n}\}\cup \{h_n\}$. Clearly, $(h_n)_n$ is a non-increasing sequence of rational numbers. Moreover, since $f$ is strictly decreasing, for all $\xi>\htop(X)$ there exists $n\in\bN$ such that $h_n<\xi$. This shows that $(h_n)_n$ converges to $\htop(X)$ from above. We conclude that $\htop(X)$ is computable based on oracle access to $\cG$. 
\end{proof}
We note that the key feature of Proposition \ref{prop52} is that we do  not require  oracle access to the language of $X$. Under the assumption $h_{\rm con}(X)>h_{\rm res}(X)$ having oracle access to the generating set and language of $X$, the computability of $\htop(X)$ has already been established in \cite{burr_computability_2022}.

\section{Examples}\label{examples}

\subsection{$\beta$-shifts}
A well-studied class of shift spaces is given by the $\beta$-shifts, which arise from the interval expanding maps $T_\beta\colon [0,1) \to [0,1)$ defined by
\[
T_\beta(x) = \beta x \bmod 1.
\]
These maps correspond to non-integer representations of real numbers. When $\beta$ is an integer, the resulting expansions reduce to the familiar $n$-ary representations (for example, binary when $\beta=2$, decimal when $\beta=10$, and so on). The maps $T_\beta$ are known as $\beta$-transformations, and their ergodic properties were first studied by Rényi and Parry \cite{renyi_representations_1957, parry_-expansions_1960}. We briefly introduce $\beta$-shifts below and then verify that Theorem~\ref{thmmain2} applies to this class of systems.

Let $d$ be the number of pieces of monotonicty of $T_\bt,$ that is $d=\ceil{\bt}.$ By using the partition 

\[
J_0 = \left[ 0,\frac{1}{\bt}\right), J_1 = \left[\frac{1}{\bt},\frac{2}{\bt}\right),J_{d-1}=\left[\frac{d-1}{\bt},1 \right) 
\]
we can associate $T_{\bt}$ to a (one-sided) subshift $X(\bt)$ of $\{0, \dots, d\}^{\NN}$, known as a $\bt$-shift. As with all shifts, they can be represented by a directed labeled graph, and by taking bi-infinite paths on this graph we can obtain a two-sided coded shift $\hat{X}(\bt) \subset \{0, \dots , d\}^{\ZZ}.$ We detail this procedure (see \cite{kucherenko_ergodic_2024}) here before applying our results. Every number $p \in [0, 1)$ is associated to a sequence $\Om(p) \in \{0, \dots , d\}^{\NN}$ defined by $\Om(p)_n = i \Longleftrightarrow T^{-n}_\bt(p)\in J_i.$ The shift $X(\bt)$ is then the closure of $\Om([0, 1)).$

Now let $d_\beta(1)=(\oldepsilon_k)$ be the $\beta$-expansion of 1, i.e. $d_\beta(1) = \lim_{p\to1} \Om(p).$ The map $\Om$ intertwines the linear ordering on $[0, 1)$ with the lexicographic ordering on $\{0, \dots, d\}^{\NN}$, which we denote by $\preceq.$ The sequences in $\{0, \dots, d\}^{\NN}$ that belong to $X(\bt)$ can be characterized in the following way \cite{parry_-expansions_1960}:
\begin{equation}\label{eqbetab}
X_\bt = \{x\in \{0,\dots, d-1 \}^{\NN}\colon \sg^k(x)\preceq d_\beta(1) \text{ for all }k\ge 0 \}.
\end{equation}
Next we define the two-sided $\beta$-shift $\hat{X}_\beta$ as the coded shift generated by
\begin{align}\label{beta_generators}
    \cG(\bt) &= \{\oldepsilon_1\cdots \oldepsilon_ji \colon i<\oldepsilon_{j+1} \}\cup\{0,\dots,\floor{\beta}-1 \}.
\end{align}
Hofbauer showed that $X(\bt)$ has a unique MME \cite{hofbauer_-shifts_1978}. In \cite{kucherenko_ergodic_2024} it is shown that $\hat{X}_\beta$ has a unique MME which is $\cG(\bt)$-Bernoulli with $p_g = \beta^{-|g|}$ and Vere-Jones parameter $\kp = \sum_{g\in \cG(\bt)}|g|p_g.$ The following result establishes the computability of the unique MME of the $\beta$-shift based on oracle access to its generators $\cG(\beta)$. We stress that while we technically only require the generators as input, the algorithm also uses specific "knowledge" about being a $\beta$-shift, e.g.,  \eqref{eqbetab} and \eqref{beta_generators}.

\begin{corollary}\label{corbeta1}Let $\beta>1$ be a non-integral real number. 
    Let $\hat{X}_\bt$ be the 2-sided $\bt$-shift, with generating set as in equation \eqref{beta_generators}. Then $\beta$ and the unique MME of $\hat{X}_\bt$ are computable based on oracle access to $\cG(\beta)$.
\end{corollary}

\begin{proof}
Given the oracle of $\cG(\beta)$, the computability $\cL(\hat{X}_\bt)$ (as well as of $b(\beta)$) follows from \eqref{eqbetab} and \eqref{beta_generators}. It is shown in \cite{burr_computability_2022} that $h_{\rm res}(\hat{X}_\bt)=0$ and that $h_{\rm top}(\hat{X}_\bt)$ is computable from the oracles of $\cG(\beta)$ and $\cL(\hat{X}_\bt)$. Since $h_{\rm top}(\hat{X}_\bt)=\log \beta$ we conclude that $\beta$ is computable.
Note that
$\hat{X}_\bt$ is of bounded growth $b(\cG)=|\cG_1|$.
  Thus, the measure $\mu$ in Theorem \ref{thmmain3} is the unique MME of $\hat{X}_\bt$ and is computable.
    \end{proof}
\begin{corollary}Let $\beta>1$ be a non-integral real number given by an oracle. 
    Let $X_\bt$ be the one-sided $\bt$-shift, with generating set as in equation \eqref{beta_generators}. Then $\beta$ and the unique MME of $X_\bt$ are computable based on oracle access $\cG(\bt).$
\end{corollary}

\begin{proof}That $\beta$ is computable is shown in the proof of Corollary \ref{corbeta1}. Recall that $X_\beta$ and $\hat{X}_\bt$ have the same language.
Since the unique MME of $X_\beta$ assigns the same measure to cylinders as the unique MME of the $2$-sided $\beta$-shift, the computability of the unique MME of  $X_\beta$ follows from Corollary \ref{corbeta1} and a variant of Lemma \ref{mu_computable} for one-sided shift spaces.
    \end{proof}

\subsection{$S$-gap and generalized gap shifts}\label{sec62}

Another class of examples with interesting dynamical properties is given by $S$-gap shifts. They consist of binary sequences and are defined as follows: Let $S\subset\NN\cup\{ 0\}.$ The $S$-gap shift given by $S$ is the coded shift generated by 
\[
\cG = \{0^s1\colon s\in S \}.
\]
One generalization of $S$-gap shifts are so-called generalized gap shifts on the alphabet $\{0,\dots,d \},$ where $d\ge 1.$ Let $S_0,\dots,S_{d-1}$ be subsets of $\NN\cup \{ 0\}.$ We will denote the group of permutations on $\{0,\dots,d-1 \}$ as $G_d$ and fix $\Pi\subset G_d.$ Note that we do not require $\Pi$ to be a subgroup. The generalized gap shift associated to $S_0,\dots,S_{d-1}$ and $\Pi$ is the coded shift generated by 

\begin{align}\label{gen_S_gap}
    \cG = \{\pi(0)^{s_{\pi(0)}}\cdots \pi(d-1)^{s_{\pi(d-1)}}d\colon \pi\in\Pi, s_j\in S_j. \}    
\end{align}
The generating set for generalized gap shifts above uniquely represents its concatenation set \cite{burr_computability_2022}. Using this result we can show the following.

\begin{corollary}\label{corgengap}
    Let $X$ be the generalized gap shift associated to $S_0,\dots,S_{d-1}$ and permutation set $\Pi.$ Then $X$ has a unique MME $\mu_{\rm max}$ which is $\cG$-Bernoulli. Moreover, the Vere-Jones parameter $\kp$ and the measure $\mu_{\rm max}$ are both computable based on having oracle access to $S_1,\dots,S_{d-1},$ and $\Pi.$
\end{corollary}

\begin{proof}
That $X$ has a unique MME $\mu_{\rm max}$ which is $\cG$-Bernoulli was already proven in \cite{kucherenko_ergodic_2024}. Here we also use $h_{\rm res}(X)=0$ (see \cite{burr_computability_2022}).
    Our goal is to apply Theorem \ref{thmmain2} to show that $\kp$ and $\mu_{\rm max}$ are computable. To do this we must show that we can recursively compute $\cL(X)$ and $\cG$ from inputs $S_1,\dots,S_{d-1},$ and $\Pi.$ By its definition, we can clearly compute $\cG.$ For $\cL(X),$ note that $\cG$ generates the language of $\Xcon,$ which in turn allows us to show that we can compute $\cL(X)$ recursively.
    A simple way to obtain a rough but sufficient upper bound for $|\cG_n|$ is the following:
    \begin{align*}
        |\cG_n| &\le |\Pi|\left|\left\{(s_0,\dots,s_{d-1})\in\{0,\cdots, n-1 \}^{d-1} \colon s_0+\cdots+s_{d-1}=n-1  \right\}\right| \\
            &\le d!\cdot  n^{d-1}.
    \end{align*}
    Thus we can estimate 
    \begin{equation}\label{eqwert}    
   \frac{1}{n}\log|\cG_n|\le \frac{\log d!}{n}+\frac{(d-1)\log n}{n}.
    \end{equation}
    Note that the right hand side in \eqref{eqwert} is decreasing in $n$ and converges to zero as $n\to\infty$. Therefore, by using  that $h_{\rm top}(X)$ is computable, we can compute $N$ such that $\frac1n\log |\cG_n|<\frac12\htop(X)$ for all $n> N$. Recall that $\frac{1}{n}\log|\cG_n|<\htop(X)$ for all $n\in\bN$. Thus, by computing $\htop(X)$ at sufficient precision, we can compute $\epsilon>0$ satisfying
     $$\htop(X)-r(\cG)>\max\left\{\frac12\htop(X),\htop(X)-\frac{1}{n}\log|\cG_n|:n\le N\right\}=\epsilon.$$
    The result now immediately follows from Theorem \ref{thmmain2}.
    \end{proof}

It should be noted that the $S$-gap shift is a generalized gap shift with $\Pi=\{{\rm id} \}$ and $d=1,$ so Corollary \ref{corgengap} also applies to $S$-gap shifts. 

\subsection{Dyck Shift}
The standard Dyck shift is an example of a coded shift with precisely two distinct ergodic MMEs, which we denote by $\mu_1$ and $\mu_2.$ The alphabet is $\{(,),[,] \}$ and the canonical generating set is $\cG = \bigcup_{n=1}^\infty W_{n},$ where $W_n$ is defined recursively as follows:
\begin{align*}
    W_1 &= \{(),[] \} \\
    W_n &= \left\{(w), [w]\colon w\text{ is a concatenation of words from }\bigcup_{k=1}^{n-1}W_k\text{ and } |w| = 2n-2 \right\}.
\end{align*}
In \cite{kucherenko_ergodic_2024} the authors define alternative generating sets $\cG^{(1)}$ and $\cG^{(2)}$  which uniquely represent $\Xcon(\cG^{(1)})$ and $\Xcon(\cG^{(2)})$, respectively. These generating sets are defined as follows:

\begin{align*}
    \cG^{(1)} &= \cG \cup \{(,[ \} \\
    \cG^{(2)} &= \cG \cup \{),] \}.
\end{align*}
Furthermore, it is shown in \cite{kucherenko_ergodic_2024} that $\hcon(X(\cG^{(i)})) = \hres(X(\cG^{(i)})) = \htop(X) = \log 3,$ and that $\mu_{i}\left(\Xcon\left(\cG^{(i)}\right)\right) = 1$ for $i\in\{1,2\}.$ It is also shown that the measures $\mu_i$ are $\cG^{(i)}$-Bernoulli with $p_g = 3^{-|g|}$ and $\kp=2.$ Taking $\kp=2$ as input to the algorithm immediately implies the $\cG^{(i)}$-computability of $\kp$.
It now follows from Theorem \ref{thmmain} that both $\mu_1$ and $\mu_2$ are computable from oracles of $\cG^{(1)}$ and $\cG^{(2)},$ respectively.
We note that in this case the algorithm also uses  the (external) information $p_g=3^{-|g|}$ and $\kp=2$.\\[0.1cm]
We remark that rather than relying on $\kp=2$ to establish the $\cG^{(i)}$-computability of $\kp$ we could have also worked with the estimates in \cite{kucherenko_ergodic_2024} to compute an $0<\epsilon<h_{\rm con}(X)-r(\cG)$ and apply Lemma
\ref{c_computable} to derive the $\cG^{(i)}$-computability of $\kp$.

\end{document}